\documentclass{article}

\usepackage{PRIMEarxiv}

\usepackage{amsmath,amssymb,amsfonts,amsthm,bbm,mathrsfs,verbatim} 
\usepackage[misc]{ifsym}
\usepackage{bbm}
\usepackage{graphics}                 
\usepackage{graphicx}
\usepackage{subfig}
\usepackage{float}
\usepackage{dsfont}
\usepackage{array}
\usepackage{fancyhdr}       
\usepackage{mathtools}
\usepackage{blkarray}
\usepackage{makecell}
\usepackage{color}                    
\usepackage{hyperref}                 
\usepackage{makeidx}
\usepackage{caption}
\usepackage{bbm}
\usepackage{algorithm}
\usepackage{algpseudocode}
\usepackage{authblk}
\usepackage{amsmath} 
\usepackage{booktabs}

\DeclarePairedDelimiter\norm{\lVert}{\rVert}

\DeclareUnicodeCharacter{02B9}{'}

\newtheorem{theorem}{Theorem}[section]
\newtheorem{proposition}[theorem]{Proposition}

\newtheorem{lemma}[theorem]{Lemma}
\theoremstyle{definition}
\newtheorem{remark}[theorem]{Remark}
\newtheorem{definition}[theorem]{Definition}
\newtheorem{example}[theorem]{Example}

\newtheorem{notation}[theorem]{Notation}

\usepackage{neuralnetwork}

\usepackage{tikz}
\usepackage{listofitems} 
\usetikzlibrary{arrows.meta} 
\usepackage[outline]{contour} 

\usetikzlibrary{positioning, calc, arrows.meta}

\tikzset{>=latex} 
\usepackage{xcolor}
\colorlet{myred}{red!80!black}
\colorlet{myblue}{blue!80!black}
\colorlet{mygreen}{green!60!black}
\colorlet{myorange}{orange!70!red!60!black}
\colorlet{mydarkred}{red!30!black}
\colorlet{mydarkblue}{blue!40!black}
\colorlet{mydarkgreen}{green!30!black}
\tikzstyle{node}=[thick,circle,draw=myblue,minimum size=22,inner sep=0.5,outer sep=0.6]
\tikzstyle{node in}=[node,green!20!black,draw=mygreen!30!black,fill=mygreen!25]
\tikzstyle{node hidden}=[node,blue!20!black,draw=myblue!30!black,fill=myblue!20]
\tikzstyle{node convol}=[node,orange!20!black,draw=myorange!30!black,fill=myorange!20]
\tikzstyle{node out}=[node,red!20!black,draw=myred!30!black,fill=myred!20]
\tikzstyle{connect}=[thick,mydarkblue] 
\tikzstyle{connect arrow}=[-{Latex[length=4,width=3.5]},thick, mydarkblue,shorten <=0.5,shorten >=1]
\tikzset{ 
	node 1/.style={node in},
	node 2/.style={node hidden},
	node 3/.style={node out},
}

\newcommand\grad[1]{\nabla \, #1}

\begin{document} 
\title{Fredholm Neural Networks }

\author{\textbf{Kyriakos Georgiou}\thanks{Dipartimento di Ingegneria Elettrica e delle Tecnologie dell'Informazione, Università degli Studi di Napoli ``Federico II'', Naples, Italy. kyriakoskristofer.georgiou@unina.it}, \,\,\,\,   
\textbf{Constantinos Siettos}\thanks{Dipartimento di Matematica e Applicazioni ``Renato Caccioppoli'', Università degli Studi di Napoli ``Federico II'', Naples 80126, Italy. constantinos.siettos@unina.it (corresponding author)}, \,\,\, \, 
\textbf{Athanasios N. Yannacopoulos}\thanks{Department of Statistics and Stochastic Modelling and Applications Laboratory, Athens University of Economics and Business, Athens, Greece. ayannaco@aueb.gr (corresponding author)}
}

\date{}
\maketitle

\begin{abstract}
Within the family of explainable machine-learning, we present Fredholm neural networks (Fredholm NNs): deep neural networks (DNNs) architectures motivated by fixed- point iteration schemes for the solution of linear and nonlinear Fredholm integral equations (FIEs) of the second kind. We also show how the proposed framework can be also used for the solution of inverse problems. Applications of FIEs include the solution of ordinary, as well as partial differential equations (ODEs, PDEs) and many more. We first prove that Fredholm NNs provide accurate solutions. We then provide insight into the values of the hyperparameters and trainable/explainable weights and biases of the DNN, by directly connecting their values to the underlying mathematical theory. For our illustrations, we use Fredholm NNs to solve both linear and nonlinear problems, including elliptic PDEs and boundary value problems. We show that the proposed scheme achieves significant numerical approximation accuracy across both the domain and boundary. The proposed methodology provides insight into the connection between neural networks and classical numerical methods, and we posit that it can have applications in fields such as Uncertainty Quantification (UQ) and explainable artificial intelligence (XAI). Thus, we believe that it will trigger further advances in the intersection between scientific machine learning and numerical analysis.

\end{abstract}	

\section{Introduction}
The bridging of machine learning and numerical analysis/scientific computing, for the solution of the forward and inverse problem in dynamical systems in the form of ODEs and/or PDEs, can be traced back to the '90s. An intrinsic, such effort, can be found in \cite{rico1992discrete}, where the authors constructed a feedforward artificial neural network (ANN) that implements the fourth order-Runge Kutta method. The proposed Runge-Kutta NN is capable  of both fitting the right-hand side of a continuous dynamical system in the form of ODEs and solving the identified ODEs in time. In \cite{gonzalez1998identification}, one can find the ``ancestor'' of convolutional neural networks (CNNs) for the identification of partial differential equations using five-points stencils of finite differences. In \cite{meade1994numerical} and  \cite{lagaris1998artificial} one can trace the ``ancestors'' of physics-informed neural networks \cite{raissi2019physics}. In the first of these works \cite{meade1994numerical} a neural network with constrained weights and biases is proposed for the solution of linear ODEs, while in the second one \cite{lagaris1998artificial} feedforward neural networks are used to solve nonlinear problems of PDEs. Finally, in \cite{siettos2002semiglobal,siettos2002truncated,alexandridis2002modelling} one can find a forebear of explainable/interpretable artificial intelligence, where fuzzy systems, neural networks and Chebyshev series are used for the construction of normal forms of dynamical systems.\par
More recently, due to both theoretical and technological advancements, the field of scientific machine learning has gained immense traction and interest from researchers and practitioners for the solution of both the forward and inverse problems in dynamical systems and beyond (e.g., for optimization and design of controllers). Specifically, the introduction of PINNs\cite{raissi2019physics} and neural operator networks such as DeepONet \cite{lu2021learning}, Fourier neural networks \cite{li2020Fourier}, graph-based neural operators \cite{li2020multipole,kovachki2023neural}, random feature models \cite{nelsen2021random}, and more recently, random operator (shallow) networks \cite{fabiani2024randonet} provide an arsenal of such scientific machine learning methods (for a review see, e.g., \cite{karniadakis2021physics}). 
In addition, other methods have also been developed, such as Legendre neural networks (\cite{mall2016application}), and simulation-based techniques (\cite{frey2022deep}). These approaches have been applied to problems across many domains, where quantities of interest can be represented as solutions to Ordinary/Algebraic differential equations (ODEs, DAEs) \cite{lu2021deepxde,ji2021stiff,kim2021stiff,de2022physics,fabiani2023parsimonious} and Partial Differential Equations (PDEs) \cite{han2018solving,raissi2019physics,pang2019neural,lee2020coarse,lu2021deepxde,chen2021solving,fabiani2021numerical,calabro2021extreme,dong2021local,dong2021modified,galaris2022numerical}, as well as to Partial Integro-differential Equations (PIDEs) \cite{yuan2022pinn, georgiou2024deep,fabiani2024task}. 
Similarly, seminal work has been done when considering Integral Equations (IEs) such as the Fredholm and Voltera IEs  with applications in various domains (see e.g., \cite{mikhlin2014integral}). Furthermore, IEs can be connected to ODEs and PDEs. In particular, the connection to PDEs is established via the Boundary Integral Equation (BIE) method (see e.g., \cite{kellogg2012foundations,hsiao2008boundary}). This connection has been considered in the context of neural networks in \cite{lin2021binet}, where the authors developed the BINet. This method considers a DNN which is trained to learn the boundary function using either the single or double layer potential representation. Subsequently, the function is integrated to obtain the solution of the PDE. Indeed, as the authors describe, the use of the BIE is very useful since the dimension of the problem becomes smaller (having to only estimate the auxiliary function along the boundary). In  \cite{effati2012neural, guan2022solving, jafarian2013utilizing, qu2023boundary}, the authors use loss functions based on the residual between the model prediction and integral approximations to estimate the solution to FIEs using the standard learning approach for DNNs. Recently, in \cite{solodskikh2023integral}, the authors presented Integral Neural Networks, which were applied to computer vision problems, and were based on the observation that a fully connected neural network layer essentially performs a numerical integration. \par
Naturally, of vital importance is the error analysis of the neural network models. In particular, when considering ODEs and PDEs with boundary conditions, it has been reported that such models can exhibit large errors near the boundaries (see \cite{georgiou2024deep} when considering Kolmogorov backward equations in credit risk) and, more generally, in problems with more complex dynamics, as displayed in \cite{krishnapriyan2021characterizing}. For a rigorous analysis of PINN errors we refer the reader to e.g., \cite{de2022error, doumeche2023convergence}. However, despite recent works that focus on the mathematical analysis of the performance of neural network models, model architecture selection and hyperparameter tuning still relies mainly on ad hoc/heuristic approaches (see for example \cite{yu2020hyper,dong2021modified,galaris2022numerical,fabiani2023parsimonious,bolager2024sampling}). This, along with the fact that training is always done using standard loss function optimization (which may even fail in some cases, as shown in \cite{wang2022and}), leads to weights and biases that have practically no interpretability.
This issue is amplified particularly in the case of DNNs, where the complexity of the inner calculations creates the ``black-box'' effect, meaning that researchers and practitioners are not able to quantify errors and the effect of hyperparameters a-priori, in the same way that stability and monotonicity can be analyzed in numerical schemes such as finite difference and finite elements schemes. Very recently, in  \cite{doncevic2024recursively} based on the concept of Runge-Kutta NNs presented in \cite{rico1992discrete}, the authors propose a recursively recurrent neural network (R2N2) superstructure that can learn numerical fixed point algorithms such as Krylov, Newton-
Krylov and RK algorithms. \par  
Here, motivated by the above challenges and advancements,  we study the connection between Fredholm integral equations and DNNs focusing on explainable machine learning.  
Specifically, we show that, by considering a discretized version of IEs, we can write the process of successive approximations as a feedforward, fully connected DNN, which resembles the fixed point iterations required reaching the solution of FIEs. The connection between integral discretization and the calculation done within the inner layers of a DNN is known (see e.g., the recent paper \cite{solodskikh2023integral}). However, despite this connection, all literature and applications use a model training step (via an appropriate loss function), which may suffer from the aforementioned lack of interpretability. Here, we focus on this connection from a mathematical standpoint, and study how DNNs can be used to solve IEs, corresponding to ODEs and PDEs by taking advantage of their connection to IEs. This allows us to avoid the standard training step; we construct a DNN, with known activation functions, weights and biases, tailor-fitted to the form of the IE via the fixed point approximation theory. In our DNN, the number of hidden layers, nodes per layer and activation function have a straightforward  interpretation related to the theory of fixed point schemes. We refer to the proposed neural networks as Fredholm neural networks (Fredholm NNs), and  show that they can be used easily and efficiently to solve FIEs in various settings.\par 
Under this framework, Fredholm NNs bridge the gap between rigorous numerical schemes and ``black box'' surrogate DNN models. We note that, in this paper, we focus on  the presentation of the methodology and provide the theoretical background for the construction of the Fredholm NNs and provide examples for the forward and inverse problems. The relevant codes are available at \href{https://github.com/kygeor/Fredholm_Neural_Networks/tree/main}{GitHub: kygeor/Fredholm Neural Networks}. Our approach contributes to the growing field of explainable/interpretable artificial intelligence (XAI) and interpretable modelling, and also creates a path for further research such as the analysis and construction of other DNN architectures (e.g., Recurrent neural networks) which can be used for non-linear IEs, as well as analyzing the errors in DNN models and developing alternative constructions and training techniques.

\section{Methodology}
As mentioned, our aim is to explore the connection between Fredholm integral equations (FIEs) and DNNs. As such, some of the theory below can also be found in the aformentioned literature pertaining to connections between IEs and DNNs. However, for the completeness of the presentation, we provide some preliminaries regarding the solutions for FIEs.

\subsection{Solving Fredholm Integral Equations}
We will consider both linear and non-linear 
integral equations. 
The aim of this section is to outline important results regarding such problems. These constitute the basis for the proposed explainable DNNs  for the approximation of their solutions. 
\begin{notation}
\par Throughout this paper, 
we consider the Banach space $\mathcal{X}:= \mathcal{C}\big(D, \mathbb{R} \big)$, endowed with the norm $\norm{f(x)} = \sup_{x \in \mathcal{D}} |f(x)|$, with $x \in \mathbb{R}^d$ and $\mathcal{D} \subset \mathbb{R}^d$ a compact domain, and the Hilbert space ${\mathcal H}:=L^{2}(D ; R)$ endowed with the inner product $\langle f, g \rangle =\int_{D} f(x) g(x) dx$. We remark, however, that the results below are generalizable to other spaces, e.g., $\mathcal{C}(\mathcal{D}, \mathbb{R}^n)$ or Lebesgue spaces.
\end{notation}
    
 With the above notation at hand, let $K:\mathcal{D} \times \mathcal{D} \rightarrow \mathbb{R}$ be a kernel function and $g: \mathcal{D} \rightarrow \mathbb{R}$. We will be studying linear FIEs of the second kind, which are of the form:
\begin{eqnarray}\label{ie}
	f(x) = g(x) + \int_{\mathcal{D}}K(x,z) f(z)dz,
\end{eqnarray}
as well as the non-linear counterpart,
\begin{eqnarray}\label{nl-ie-def}
    f(x) = g(x) + \int_{\mathcal{D}}K(x,z) G(f(z))dz,
\end{eqnarray}
for some function $G: \mathbb{R} \rightarrow\mathbb{R}$ considered to be a Lipschitz function, particularly focusing on cases where the integral operator is either contractive or non-expansive. We remind the reader of these properties in the definition below. 
\begin{definition}
    Consider an operator $\mathcal{T}: \mathcal{X} \rightarrow \mathcal{X}$. We say that $\mathcal{T}$ is a $q-$contraction if:
    \begin{eqnarray}
        \norm{\mathcal{T}f_1 - \mathcal{T}f_2} \leq q \norm{f_1 - f_2},
    \end{eqnarray}
    for some $q < 1$. If $q=1$, then we say the operator is non-expansive.
\end{definition}

\par Throughout this paper, we will use the operator $\mathcal{T}$, also denoted as ${\cal T}^{K,g}$, defined by:
    \begin{eqnarray}\label{OPERATOR}
        (\mathcal{T}f)(x)  = g(x) +  \int_{\mathcal{D}}K(x,z)f(z)dz.
    \end{eqnarray}
For brevity, we will also define the integral operator $\mathcal{I}$: 
\begin{eqnarray}
	(\mathcal{I}f)(x) = \int_{\mathcal{D}}K(x,z)f(z)dz.
\end{eqnarray}
When $\mathcal{T}$ is a contraction, it is well known that the FIE can be solved using the Neumann series $f(x) = \sum_{n=0}^{\infty} (\mathcal{I}^{(n)}g)(x)$, whose convergence is guaranteed by the Banach Fixed-Point Theorem. Simultaneously, such equations can also be solved numerically utilizing appropriate discretization schemes, such as the Nyström method. In this work, we will focus on the fixed point approach, and illustrate its connection to Neural Networks. At this point, we remind the reader of the Krasnoselskii-Mann method (see e.g., \cite{bauschke2017correction} and \cite{kravvaritis2020variational} for further details regarding fixed point methods). In the particular case, when $q=1$, we will extend our approach to the Hilbert space ${\cal H}$ (note that for $\mathcal{D}$ compact ${\cal X} \subset {\cal H}$). Moreover, there are versions of Krasnoselskii-Mann (KM) theorem that are valid in the Banach space, but this is not required here \cite{reich1979weak,zhang2021generalized}.
We will start by considering the linear case.\\

\begin{proposition}[Krasnoselskii-Mann (KM) method]\label{km-method}  Consider the 
 linear FIE (\ref{ie}) defined by a non-expansive operator $\mathcal{T}$, on ${\cal H}$. Furthermore, consider a sequence $\{\kappa_n\}, \kappa_n \in (0,1]$ such that $\sum_n \kappa_n(1-\kappa_n) = \infty$. Then the iterative scheme:
\begin{eqnarray}\label{km-it}
    f_{n+1}(x) = f_n(x) + \kappa_n(\mathcal{T}f_n(x) -f_n(x)) = (1-\kappa_n)f_n(x) + \kappa_n \mathcal{T} f_n(x),
\end{eqnarray}
with $f_0(x) = g(x)$, converges to the fixed point solution of the FIE, $f^{*}(x)$.
\end{proposition}
For a proof of the above see e.g., in \cite{kravvaritis2020variational}.\par
It is straightforward to see that, when $\mathcal{T}$ is a contraction, we can set $\kappa_n =1$, for all $n$, in which we can work on the original Banach space $\cal{X}$ and obtain the iterative process:
\begin{eqnarray}\label{iterations}
	f_n(x)= g(x) + (\mathcal{I} f_{n-1})(x), \,\,\ n \geq 1,
\end{eqnarray}
which converges to the fixed point solution. This is often referred to as the method of successive approximations, which is easily seen to coincide with the Neumann series expansion. Both (\ref{km-it}) and (\ref{iterations}) will be useful for the methods we will apply in the following sections.\par 
This category of equations will allow us to develop a straightforward connection with standard deep neural networks, thus forming the basis of this novel way of constructing and interpreting such models. 

\subsection{Connection between linear FIEs and Deep Neural Networks}
In this section, we present the relationship between linear FIEs and DNNs and describe their construction motivated by the fixed point iteration process. In particular, we demonstrate a tailor-made/explainable deep neural networks (DNNs) that can approximate to any given accuracy the solution of FIEs. Importantly, the structure and the weights of the DNN are shaped by the form of the FIEs and the required approximation accuracy.

Here,  the aim is to estimate the solution of FIEs, rather than learning the integral operator. We do this by showing that the parameters of the DNN can be interpreted entirely based on the theory of fixed point methods. In this way, we are able to establish a connection between the mathematical theory of FIEs and DNNs, creating an explainable/interpretable DNN-based numerical estimation framework. \par 
Here, we introduce the following definition. 

\begin{definition}[$M-$layer fixed point estimate] \label{dd}
Consider a function $f$ satisfying the FIE (\ref{ie}). Consider a $z-$grid $Z = \{z_1, \dots z_N \}$, with $|Z| = N$, such that $x \in Z$. Furthermore, consider the following operator representing an iteration of KM-algorithm (rearranged) using the discretized version of the integral operator, $\mathcal{T}$, i.e.,:
\begin{eqnarray}\label{disc-op}
	(\mathcal{T}_Z f) (x) := \kappa_n g(x) + \sum_{j =1}^Nf(z_j)\Big( K(x,z_{j})\kappa_n\Delta z + (1-\kappa_n) \mathbbm{1}_{\{z_j = x\}}\Big).
\end{eqnarray}
Then, we define the $M-$layer fixed point estimate of the solution to the FIE as the following representation arising from applying the iterative algorithm (\ref{km-it}):
\begin{eqnarray}\label{discr_def}
	f_M(x) = \underbrace{\mathcal{T}_{Z}\Big(\mathcal{T}_{Z}\big(\dots (\mathcal{T}_{Z}g)(x)\dots \big)\Big)}_{\text{$M$ times}}.
\end{eqnarray}
\end{definition}

\begin{remark}
The above proposed scheme, is motivated by the simplest way of discretizing the action of the integral operator $\mathcal{I}$, which allows us to approximate the FIE in terms of the Riemann sum:
\begin{eqnarray}\label{disc}
	f(x) \approx g(x) + \sum_{j=1}^{N}f(x_j)K(x,z_j)\Delta z.
\end{eqnarray}
Such discretization schemes have been used in the context of standard numerical analysis techniques. One of the novelties of the present work is indeed the establishment of the bridge between such approaches and DNNs.
\end{remark}
Furthermore, it is important to discuss the $z-$grid  used for the discretization.  In particular, we note that the $z-$grid  corresponds to the nodes of a DNN layer, providing interesting intuition behind the construction of the DNN and its connection to the discretization (\ref{discr_def}). 
In \cite{solodskikh2023integral}, the authors also represent integral equations as DNNs; however, they learn the weights of the integration using backpropagation techniques in contrast to our approach, which assigns the weights a-priori 
based on the form of the FIE, thus enhancing the explainability of the DNN.\par 
The following result details the construction of the proposed DNN model that we call Fredholm neural network (Fredholm NN). Note that, for simplicity in the calculations, we will assume that $x \in Z$, i.e., there exists $i\in {1,2,\dots, N}$ such that $x = z_i$ (this is further discussed in the results below).

\begin{figure}
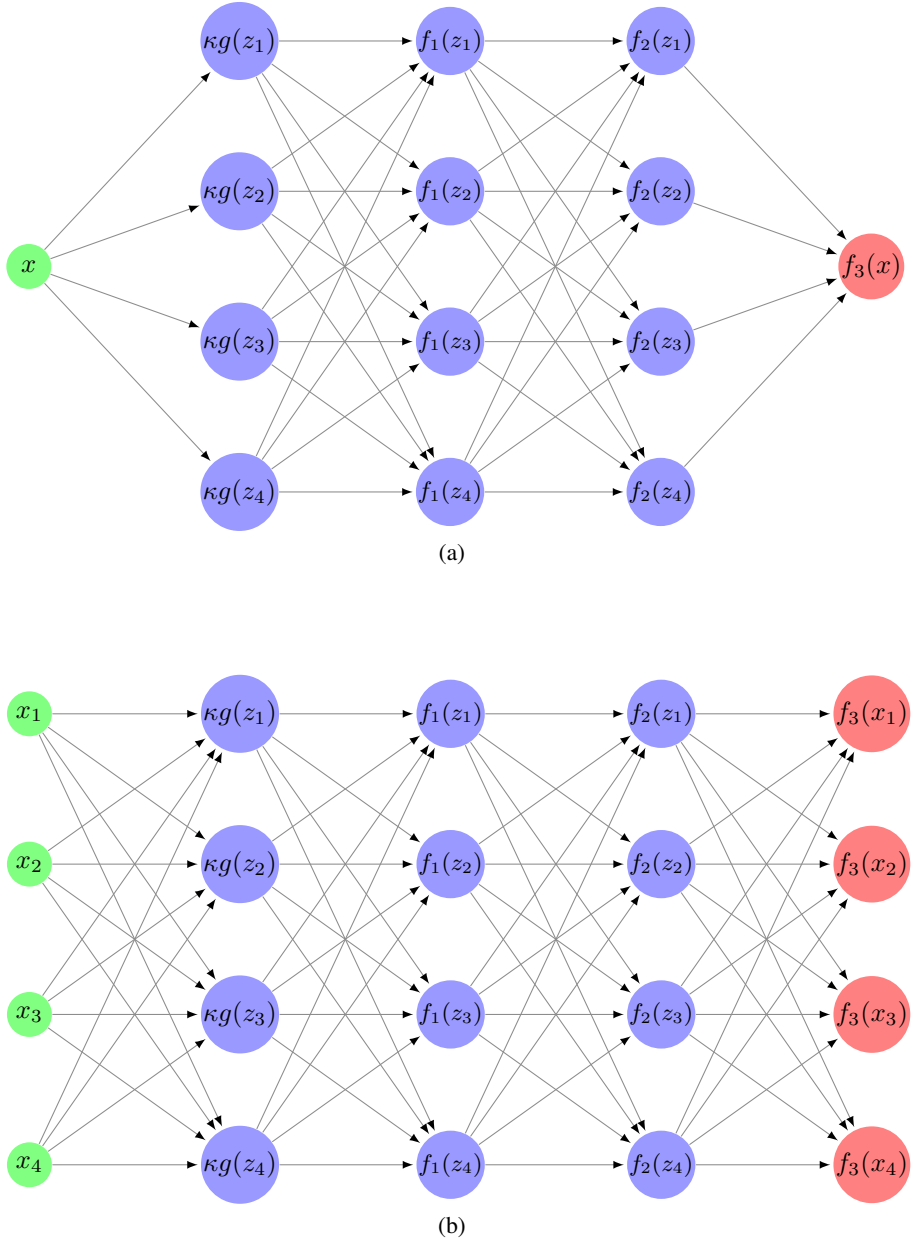
\label{fnn-versions}
\centering
\subfloat[]{\begin{neuralnetwork}[height=1., layertitleheight=3.cm, nodespacing=1.5cm, layerspacing=2.cm]
        \newcommand{\x}[2]{$\kappa g(z_#2)$}
        \newcommand{\z}[2]{$x$}
        \newcommand{\y}[2]{$f_3(x)$}
        \newcommand{\hfirst}[2]{\small $f_1(z_#2)$}
        \newcommand{\hsecond}[2]{\small $f_2(z_#2)$}
        \inputlayer[count=1, bias = false, text = \z]
        \hiddenlayer[count=4, bias=false,  text=\x] \linklayers
        \hiddenlayer[count=4, bias=false,  text=\hfirst] \linklayers
        \hiddenlayer[count=4, bias=false,  text=\hsecond] \linklayers
        \outputlayer[count=1,  text=\y] \linklayers
    \end{neuralnetwork}}

    \hfill \\
\vspace{1cm}
\subfloat[]{ \begin{neuralnetwork}[height=1.5, layertitleheight=3.6cm, nodespacing=2.0cm, layerspacing=2.8cm]
        \newcommand{\x}[2]{$\kappa g(z_#2)$}
        \newcommand{\z}[2]{$x_#2$}
        \newcommand{\y}[2]{$f_3(x_#2)$}
        \newcommand{\hfirst}[2]{\small $f_1(z_#2)$}
        \newcommand{\hsecond}[2]{\small $f_2(z_#2)$}
        \inputlayer[count=4, bias = false, text = \z]
        \hiddenlayer[count=4, bias=false,  text=\x] \linklayers
        \hiddenlayer[count=4, bias=false,  text=\hfirst] \linklayers
        \hiddenlayer[count=4, bias=false,  text=\hsecond] \linklayers
        \outputlayer[count=4,  text=\y] \linklayers
    \end{neuralnetwork}}

    \caption{Schematic of the proposed Fredholm NN as a fixed-point DNN with $M=3$ hidden layers; each node corresponds to a value on a discretized grid (we consider a constant value of the parameter $\kappa$): (a) One-dimensional output, (b) Multi-dimensional output across a grid.}
    \label{fig:DNN-grid}
\end{figure}

\begin{lemma}
[Fredholm Neural Network]\label{DNN_construction}
Consider a KM constant $\kappa = \kappa_n$ for all $n$. The $M$-layer FIE approximation $f_M(x)$ can be implemented as a Deep Neural Network with a one-dimensional input $x$, $M$ hidden layers, a linear activation function and a single output node corresponding the estimated solution $f(x)$, where the weights and biases are given by:
\begin{eqnarray}
    W_1 = \left(\begin{array}{ccc}
		\kappa g(z_1), \dots, \kappa g(z_{N})
	\end{array}\right)^{\top}, \,\,\,\,\    b_1 = \left(\begin{array}{ccc}
		0, 0, \dots, 0
	\end{array}\right)^{\top}
 \end{eqnarray} 
for the first hidden layer,  
\begin{eqnarray}	\label{inner-weight}
W_m=
\left(\begin{array}{cccc}
	K_D\left(z_1\right) & {K}\left(z_1, z_2\right)\kappa\Delta z & \cdots & {K}\left(z_1, z_{N}\right)\kappa\Delta z \\
 {K}\left(z_2, z_1\right)\kappa\Delta z  & K_D\left(z_2\right) & \cdots & {K}\left(z_2, z_{N}\right)\kappa\Delta z \\
	\vdots & \vdots & \ddots & \vdots \\
	\vdots & \vdots & \vdots & \vdots \\
	{K}\left(z_{N}, z_1\right)\kappa\Delta z & {K}\left(z_{N}, z_2\right)\kappa\Delta z & \cdots & K_D\left(z_{N}\right) 
\end{array}\right),
\end{eqnarray}
and
\begin{eqnarray}
	b_m=\left(\begin{array}{ccc}
		\kappa g(z_1), \dots, \kappa g(z_{N})
	\end{array}\right)^{\top},
\end{eqnarray}
for hidden layers $m= 2, \dots, M-1$, where $K_D\left(z\right) := {K}\left(z, z\right)\kappa\Delta z + (1-\kappa)$, and finally:
\begin{eqnarray}\label{K-layer} \label{outer-weight}
	\begin{gathered}
		W_M=\left(\begin{array}{ccc}
			{K}\left(z_1, x\right)\kappa\Delta z, \dots, {K}\left(z_{i-1},x\right)\kappa\Delta z, K_D(x), {K}\left(z_{i+1}, x\right)\kappa\Delta z, \dots,
			{K}\left(z_{N}, x\right)\kappa\Delta z
		\end{array}\right)^{\top},
	\end{gathered}
\end{eqnarray}
and $b_M =\big(\kappa g(x) \big)$, for the final layer, assuming $z_i = x$.
\end{lemma}

\begin{proof}
\par The result follows from the connection between the DNN nodes and the $z-$grid, in combination with the discretized form (\ref{discr_def}). For illustrative purposes, we rewrite this form for $M=2$:  
\begin{flalign}\label{discr_def_1}
f_2(x) \approx \kappa g(x) + \sum_{j_2 =1}^N \Big( \kappa g(z_{j_2}) + \sum_{j_1 =1}^N g(z_{j_1}) \Big( {K}(z_{j_2}, z_{j_1})\kappa \Delta z & + (1-\kappa) \mathbbm{1}_{\{z_{j_1} = z_{j_2}\}}\Big) \Big) \notag \\ & \Big( {K}(x, z_{j_2})\Delta z + ( 1- \kappa) \mathbbm{1}_{\{z_{j_2} = x\}}\Big) .
\end{flalign}
 One can easily verify that (\ref{discr_def_1}) describes the sequence of calculations that occur within the described, fully-connected feedforward DNN. Specifically, given the input values, the first layer corresponds to the initialization of the iterative method across the grid described by the nodes of the layer, i.e., $\kappa g(z_{j})$. Subsequently, we can see that the expression in the inner summation in (\ref{discr_def_1}) describes the calculation done between the hidden layers, with the resulting nodes then corresponding to the approximation obtain by the second iteration of (\ref{km-it}), across the grid. The final output corresponds to the calculation performed within the outer summation, resulting in the final approximation of the solution to the FIE, $f_M(x)$. Therefore, we obtain the weight matrices $W_1, W_2 \in \mathbb{R}^{N \times N}$, as given by (\ref{inner-weight}) and (\ref{outer-weight}), respectively, with the corresponding biases $b_1 \in \mathbb{R}^{N}$ and $b_2 \in \mathbb{R}^{N}$:
\begin{eqnarray}
   b_1 = \left(\begin{array}{ccc}
		\kappa g(z_1), \dots, \kappa g(z_{N})
	\end{array}\right)^{\top}, \,\,\,\  		b_2=\left( \kappa g(x) \right).
\end{eqnarray}
It is straightforward to generalize the above representation to $M-$layers according to (\ref{iterations}), which then provides the architecture of the remaining DNN layers. 
\par Finally, notice how in (\ref{discr_def_1}) the weights and biases from the hidden layer to the output node depend on the input value $x$, obtaining the form given in (\ref{K-layer}), thus completing the construction of the DNN model. 
 \end{proof}
With the result above, we establish the entire architecture of the DNN that estimates the solution of the FIE. We illustrate the constructed DNN (with 2 hidden layers) in Fig. \ref{fig:DNN-grid}. The proposed approach describes a way of constructing explainable numerical analysis-informed DNNs. 
As described above, the depth of the DNN is equivalent to the number of successive iterations performed to reach the fixed point solution of the IE. This provides the selection of a DNN within a required numerical approximation error, by simply examining the corresponding errors in the iterative process (see e.g., \cite{micula2023iterative} for such results). \par 
Based on this construction, we present the following main result. 

\begin{theorem}\label{FUA}
Let ${\cal Z}$ be either ${\cal X}$ or ${\cal H}$, depending on whether the operator ${\cal T}^{K, g} : {\cal Z} \to {\cal Z}$  (defined in \eqref{OPERATOR}) is a contractive or non expansive.
Consider the space $\mathcal{Z'} \subset \mathcal{Z}$ defined by $\mathcal{Z'}: = \{ f \in \mathcal{Z}: (\mathcal{T}^{K,g}f) = f, \text{ for some } K, g \}$. Then, for any $f \in {\cal Z}'$ and for any given approximation error $\epsilon$, there exists a fully connected DNN, $u(x;\mathcal{W},\mathcal{B})$, with $M$ hidden layers, whose weights and biases $\mathcal{W}$ and $\mathcal{B}$ are as given in Lemma \ref{DNN_construction}, and activation function given by:
\begin{eqnarray}
 \sigma(x) =
\begin{cases}
   \kappa g(x), \text{ for  } L = 1, \\
    x, \text { for } L \geq 2,   
\end{cases}
\end{eqnarray}
where $L$ represents the hidden layer index, and $\kappa \in \mathbb{R}$, such that $\norm{f(x) - u(x;\mathcal{W},\mathcal{B})} \leq \epsilon$. 
\end{theorem}

\begin{proof}
Let $f^*$ be the fixed point solution of the Fredholm Integral Equation \ref{ie}. Firstly, we know that $f_M\rightarrow f^*$, where $f_M$ represents the $M$-th iteration of the fixed point algorithm. Furthermore, from Proposition \ref{DNN_construction} we know that each iteration can be written as a connected layer in a Deep Neural Network, with the final output being the solution estimate at each point in a predefined grid. 
\par It now remains to show that this DNN can be used for any vector of arbitrary input values $X_O = \big( x_1, \dots, x_{N_O}\big) \in \mathbb{R}^{N_O}$. By definition, the approximation, $f(X_O) = \big( f(x_1), \dots, f(x_{N_O}) \big)^{\top}$, satisfies (\ref{ie}). Hence, we can use the discretized version of the integral operator to obtain the DNN output by adding a final layer that performs the calculation:
 \begin{eqnarray}\label{final_layer}
     f(X_O) = W_{O} f_M(Z) + g(X_O),
 \end{eqnarray}
 where we adopt the notation $g(X_O) = \big( g(x_1), \dots, g(x_{N_O})\big)$, $W_{O} \in \mathbb{R}^{ N_O \times N}$ is given by:
\begin{eqnarray}\label{we_o}
    W_O= \left(\begin{array}{ccc}
	{K}\left(z_1, x_1\right)\Delta z  & \cdots & {K}\left(z_1, x_{N_O}\right)\Delta z \\
	{K}\left(z_2, x_1\right)\Delta z & \cdots & {K}\left(z_2, x_{N_O}\right)\Delta z \\
	\vdots & & \\
	\vdots & & \\
	{K}\left(z_N, x_{N_O}\right)\Delta z & \cdots & {K}\left(z_N, x_{N_O}\right)\Delta z
\end{array}\right),
\end{eqnarray}
and $f_M(Z)$ is the $N-$dimensional output of the Fredholm NNs, $f_{M}(z_j)$, for $i = \{1,\dots, N \}$. The result of this final layer provides the estimate for the $f(X_O)$, as required. 
\end{proof}
The above is displayed in Fig. \ref{fig:ensemble-NN}. For clarity, the illustration explicitly shows how we first construct the DNN that solves the FIE along the pre-determined grid and subsequently apply the discretized operator (\ref{disc}) that provides the estimated solution for any arbitrary values. Furthermore, it is interesting to note that, since the calculations occurring within the hidden layers of the DNN, as depicted in Fig. \ref{fig:ensemble-NN}, do not depend on the input values, this approach is computationally efficient as well, since all intermediary calculations needed for the approximation across the grid can be done once. The results can then be subsequently used via the $\mathcal{T}$-mapping node to obtain the approximation for any value. This can be particularly useful for cases where many hidden layers or very fine grids are required to approximate the fixed point solution within a given margin of error. 

\begin{remark}
    From a theoretical standpoint, we also remark that ANNs with polynomial activation functions are generally not universal approximators (see e.g., \cite{sonoda2017neural, leshno1993multilayer}) for all $L^2(D)$ continuous functions. Despite this known fact, Theorem \ref{FUA} shows that for particular subspaces, such as $\mathcal{Z}'$, ANNs with linear activation functions needing only the application of $g:\mathbb{R} \rightarrow \mathbb{R}$ as an activation to the $x$- grid created in the first layer can be universal approximators.
\end{remark}

Finally, a special case of the above is when the operator is $q-$contractive, for $q < 1$, simplifying the fixed point iterations to the form (\ref{iterations}). In such cases, we have the following result, which provides a closed form expression for the error in the Fredholm NNs, and conversely, the number of hidden layers required for a given error bound:

\begin{proposition}
Consider the case where the operator $\mathcal{T} : \mathcal{X} \to \mathcal{X}$ is a $q-$contraction (and hence $\mathcal{Z} = \mathcal{X}$). Then, for any $f \in \mathcal{Z}'$:
\begin{itemize}
    \item[(i)] The error in the M-layer Fredholm Neural Network approximation to the solution of the FIE, $f_M$, satisfies the following bound:
    \begin{eqnarray}\label{nn-error}
        \norm{f_M - f} \leq  \frac{q^M}{1-q}  \Big(\frac{D(b-a)^2}{2n} +  \norm{\mathcal{T}g - g} \Big), 
    \end{eqnarray}
where $D = \max \frac{d}{dz} \big(K(x,z)g(z)\big)$.
\item[(ii)] For a given error bound $\epsilon$, there exists a Fredholm NN with $M^*$ hidden layers, weights $\mathcal{W}$ and biases $\mathcal{B}$, $u(x; \mathcal{W}, \mathcal{B})=: f_{M^*}(x)$, such that $\norm{f_{M^*} - f} \leq \epsilon$, where:
    \begin{eqnarray}
        M^* \geq \frac{\ln \left(\epsilon(1-q) \right) - \ln \left(\Big(\frac{D(b-a)^2}{2n} +  \norm{\mathcal{T}g - g} \Big)\right)} {\ln q}.
    \end{eqnarray} 
\end{itemize}
\end{proposition}
    \begin{proof}
        This result for part $(i)$ follows from combining the known error bound arising in the method of successive approximations (see e.g., \cite{micula2023iterative}), for which we have:
        \begin{eqnarray} \label{error}
            \norm{f_M - f} \leq \frac{q^M}{1-q} \norm{\mathcal{T}f_0 - f_0}.
        \end{eqnarray}
        Hence, consider now $M$-layer approximation given by (\ref{discr_def}) and let $\mathcal{T}^M$ represent the operator given by the right-hand side. Then we have:
        \begin{eqnarray}
            \norm{f_M - f} \leq \frac{q^M}{1-q}  \norm{\mathcal{T}^{M} f_0 - \mathcal{T}f_0 + \mathcal{T}f_0 - f_0} \leq \frac{q^M}{1-q}\Big(\frac{D(b-a)^2}{2n} +  \norm{\mathcal{T}g - g} \Big).
        \end{eqnarray}
         Ther last step follows from the error arising from approximating the integral with the left Riemann and from the error in the successive approximation scheme. 
         \par The proof of $(ii)$ follows directly given Theorem \ref{FUA} and solving (\ref{error}) for the minimum number of hidden layers $M$.
    \end{proof}
\begin{figure}
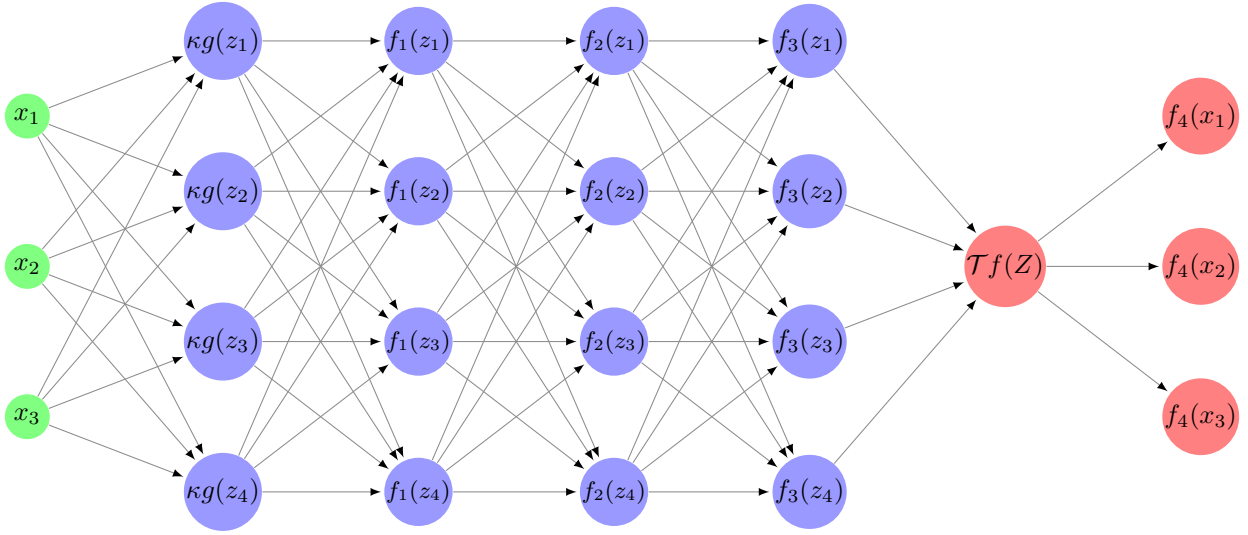

    \centering
    \begin{neuralnetwork}[height=1.5, layertitleheight=3.6cm, nodespacing=1.5cm, layerspacing=2.cm]
        \newcommand{\z}[2]{$\kappa g(z_#2)$}
        \newcommand{\y}[2]{$f_3(z_#2)$}
        \newcommand{\w}[2]{$f_4(x_#2)$}
        \newcommand{\q}[2]{$x_#2$}
        \newcommand{\nodecleartext}[2]{$\mathcal{T}f(Z)$}
        \newcommand{\hfirst}[2]{\small $f_1(z_#2)$}
        \newcommand{\hsecond}[2]{\small $f_2(z_#2)$}
        \inputlayer[count = 3, bias = false, text = \q]
        \hiddenlayer[count=4, bias=false,  text=\z] \linklayers
        \hiddenlayer[count=4, bias=false, text=\hfirst] \linklayers
        \hiddenlayer[count=4, bias=false,  text=\hsecond] \linklayers
        \hiddenlayer[count=4, bias = false, text=\y] \linklayers
        \outputlayer[count = 1, bias = false, text = \nodecleartext]\linklayers
        \outputlayer[count = 3, text = \w]\linklayers
    \end{neuralnetwork}
    \caption{Full Fredholm NN model for estimating the solution of FIEs. The first component solves the IE along the pre-defined grid, followed by the last layer to obtain the final output.}
    \label{fig:ensemble-NN}
\end{figure}

Concluding, we remark that this approach also provides versatility with respect to the fixed point algorithm and the DNN model. For example, in the case of a $q-$contraction as above, the KM algorithm results in the error bound: 
\begin{eqnarray}
    \left\|f_M-f\right\| \leq \frac{\mathrm{e}^{1-q}}{1-q}\left\|T g-g\right\| \mathrm{e}^{-(1-q) v_M},
\end{eqnarray}
where $v_0=0 \text { and } v_M=\sum_{\nu=0}^{M-1} \kappa_\nu, M \geq 1$ (see \cite{micula2023iterative} and references therein). This provides a new estimate for the DNN error if constructed to replicate the KM fixed point approach.

\subsection{Extension to non-linear FIEs}
The framework for the linear FIEs can be used to generalize the approach to non-linear integral equations. We consider the case of non-linear Fredholm integral equations of the form: 
\begin{eqnarray}\label{nl-ie}
    f(x) = g(x) + \int_{\mathcal{D}}K(x,z) G(f(z))dz,
\end{eqnarray}
for some function $G: \mathbb{R} \rightarrow\mathbb{R}$. In order to solve such IEs using the proposed framework, we will construct an iterative process that transforms the non-linear equation into a linear IE, at each step. 

\begin{proposition}\label{nl-scheme}
Let $\mathcal{Z}$ be ${\cal X}$ or ${\cal H}$, depending on whether the nonlinear operator  ${\cal T}^{K,G,g} : {\cal Z} \to {\cal Z}$ defined by
$$(\mathcal{T}^{K,G,g}f)(x) := g(x) + \int_{\mathcal{D}}K(x,z) G(f(z))dz.$$
is contracting or non expansive.
Then, the iterative scheme $f_n(x) = \Tilde{f}_n(x)$, where $\tilde{f}_n(x)$ is the solution to the linear FIE: 
    \begin{eqnarray}\label{nl-iteration}
        \Tilde{f}_{n}(x) = (\mathcal{L}\Tilde{f}_{n-1})(x) + \int_{\mathcal{D}}K(x,z) \Tilde{f}_{n}(z))dz,
    \end{eqnarray}
    and $(\mathcal{L}\Tilde{f}_{n-1})(x) := g(x) + \int_{\mathcal{D}} K(x,y)\big( G(\Tilde{f}_{n-1}(y)) - \Tilde{f}_{n-1}(y)\big)dy,$ for $n \geq 1$, converges to the fixed point   $f^*$  which is a solution of the non-linear FIE (\ref{nl-ie}).
\end{proposition}
\begin{proof}
It is straightforward to see that the fixed point of (\ref{nl-iteration}) is also a fixed point of (\ref{ie}) and that the right hand side of (\ref{nl-iteration}) maps functions $f \in \mathcal{Z}$ to $\mathcal{Z}$. Hence, by Banach's fixed point Theorem, the iteration converges to the fixed point $f^{*} \in \mathcal{Z}$.
\end{proof}

This process and the required transformation are summarized in Algorithm \ref{alg:nl-fnn} and given schematically in Fig. \ref{FNN-iteration}. With this result, we can use the FNN developed for the linear FIEs to solve equations in the form (\ref{nl-ie}). The only difference that occurs is due to the additive component in the iteration (\ref{nl-iteration}). It therefore remains to show that the proposed iteration indeed converges to the required fixed point and analyze the accompanying approximation error.

\begin{algorithm}[hbt!]
\caption{Iterative FNN for non-linear FIE.}\label{alg:nl-fnn}
\begin{algorithmic}
\State Set number of iterations $N'$.
\State Initialize $g_0(x)= g(x)$.
\While{$n \leq N'$}
    \State Step 1. Construct $ \Tilde{f}_{n}(x) = g_n(x) + \int_{\mathcal{D}}K(x,z) \Tilde{f}_{n}(z))dz$.
    \State Step 2. Solve the FIE using the Fredholm Neural Network, $u(x; \mathcal{W}, \mathcal{B}_n)$.  
    \State Step 3. Set $f_n(x) = u(x;\mathcal{W},\mathcal{B}_n)$ and $g_{n+1}(x) = g(x) + \int_{\mathcal{D}} K(x,z)\big( G(f_{n}(z)) - f_{n}(z)\big)dz$.
    \State Step 4. Set $n \gets n+1$. 
\EndWhile
\end{algorithmic}
\end{algorithm}

\begin{figure}[!ht]
\begin{center}
  \includegraphics[height=9cm, width=17cm,scale=1.0]{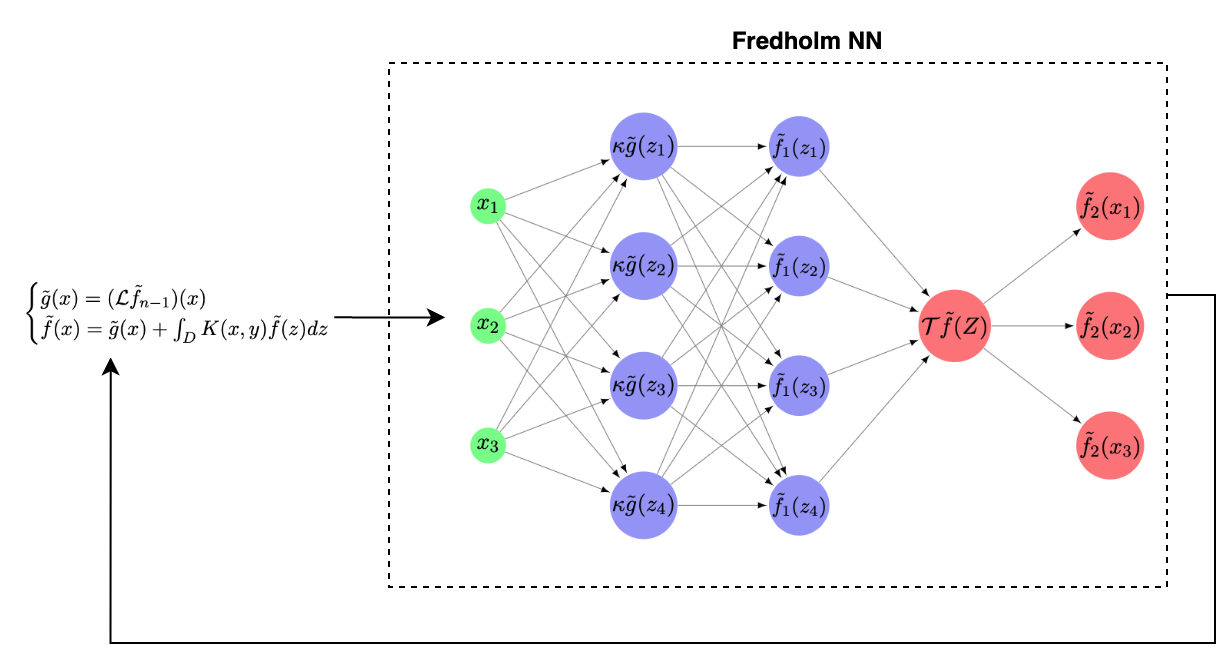} 
\end{center}
\caption{The recurrent Fredholm NNs as given in algorithm \ref{alg:nl-fnn}, for the solution of non-linear FIEs. We used the Fredholm NN to solve the FIE (\ref{nl-iteration}) at each iteration and used the result to re-define the additive term in the linear FIE.}\label{FNN-iteration}
\end{figure}

  \begin{proposition}
    Consider the non-linear Fredholm Integral Equation (\ref{nl-ie}). Then, Algorithm \ref{alg:nl-fnn} converges to the fixed point $f^* \in \mathcal{X}$. Furthermore, the approximation error satisfies (\ref{nn-error}), with $D: = \max \frac{d}{dz} \big(K(x,z)G(g(z))\big)$.
\end{proposition}
\begin{proof}
    The proof follows directly from applying Theorem \ref{FUA} to the linear FIEs (\ref{nl-iteration}) that arise from Algorithm \ref{alg:nl-fnn}. We apply an FNN to solve each iteration of the FIE that occur, thereby converging to the fixed point, as shown in Lemma \ref{nl-scheme}. 
    \par The error bound follows directly by construction of the iterative scheme (\ref{nl-iteration}).
\end{proof}  

\begin{remark}
   It is important to note that we can also utilize other ANNs architectures to solve non-linear integral equations. In particular, one can observe that a Recurrent Neural Network (RNN) creates the required structure to simulate the successive approximation calculations performed when approximating the IE (\ref{nl-ie}). The hidden layer in the RNN will have activation function $G(\cdot)$ and a single hidden state; this architecture is more specifically referred to as an Elman neural network (\cite{elman1990finding}).
\end{remark}

 Concluding, we also mention that the iterative method outlined in Algorithm \ref{alg:nl-fnn} warrants further analysis for other cases pertaining to non-contractive Integral Equations, and other generalizations of the Fredholm neural network framework. We defer this analysis for future work. 

\section{Implementation and examples}
In this section we consider various FIEs and illustrate the use of the Fredholm NNs approach for their solution. Comparisons with the exact solutions and numerical approximation errors are shown and discussed. For the examples in this section, we focus on the simpler case where we can apply (\ref{iterations}). These examples help illustrate the construction of the Fredholm NNs; then we proceed to a more complex structure using the KM algorithm in the next section. Finally, we show how the scheme can be used to solve inverse problems.
\subsection{Linear Fredholm Integral Equations}
\begin{example}\label{ex-1}
Consider the FIE given by: 
	\begin{eqnarray}
		u(x) = e^x +  \int_0^1 \frac{1}{e} u(y) dy.
	\end{eqnarray}
	With $g(x)=e^x$, it is straightforward to perform the iterative approximations to show that the fixed point solution is given by $u(x) = e^x +1$. To construct the Fredholm NNs, following Lemma \ref{DNN_construction}, for the first layer we have: 
	\begin{eqnarray}
		\begin{gathered}
				W_1=\left(\begin{array}{ccc}
				e^{y_1^{(1)}}, \dots, e^{y_{N_1}^{(1)}}
			\end{array}\right)^{\top}, 
			b_1=\left(\begin{array}{ccc}
			0, \dots, 0\end{array}\right).^{\top}
		\end{gathered}
	\end{eqnarray}
Since the kernel is constant, it is straightforward to see that the remaining weights and biases for the Fredholm NNs with $M$ layers are given by: 
\begin{eqnarray}
\begin{gathered}
	W_i=\left(\begin{array}{ccc}
		\frac{\Delta y}{e} & \ldots & \frac{\Delta y}{e} \\
		\vdots & \ldots &\vdots \\
		\frac{\Delta y}{e} & & \frac{\Delta y}{e}
	\end{array}\right),
b_i=\left(\begin{array}{ccc}
	e^{y_1}, \dots, e^{y_{N}} \end{array}\right)^{\top}, 
\end{gathered}
\end{eqnarray}
for $i = 2, \dots, M-1$, since the kernel is constant. 
We consider this architecture with constant $y-$grids of size $N=2000$, with $y_0 = 0$ and $y_N = 1$. 
\par The true solution for $x\in [0,1]$ along with the corresponding output from the proposed approach with $M = 10$ are shown in Fig. \ref{Example_2}. As seen, the two graphs are indistinguishable. In this simple example, we observe a constant error of approximately $ \norm{u_M - u^*} =8.03\cdot 10^{-4}$. To assess the effect of the DNN depth, we also plot the maximum error as a function of the hidden layers, showcasing the convergence of the method, and providing insight into the optimal number of layers required for convergence.
\end{example}

\begin{figure}[ht]
    \begin{center}
    \subfloat[]{\includegraphics[height=55mm, width=80mm,scale=1.0]{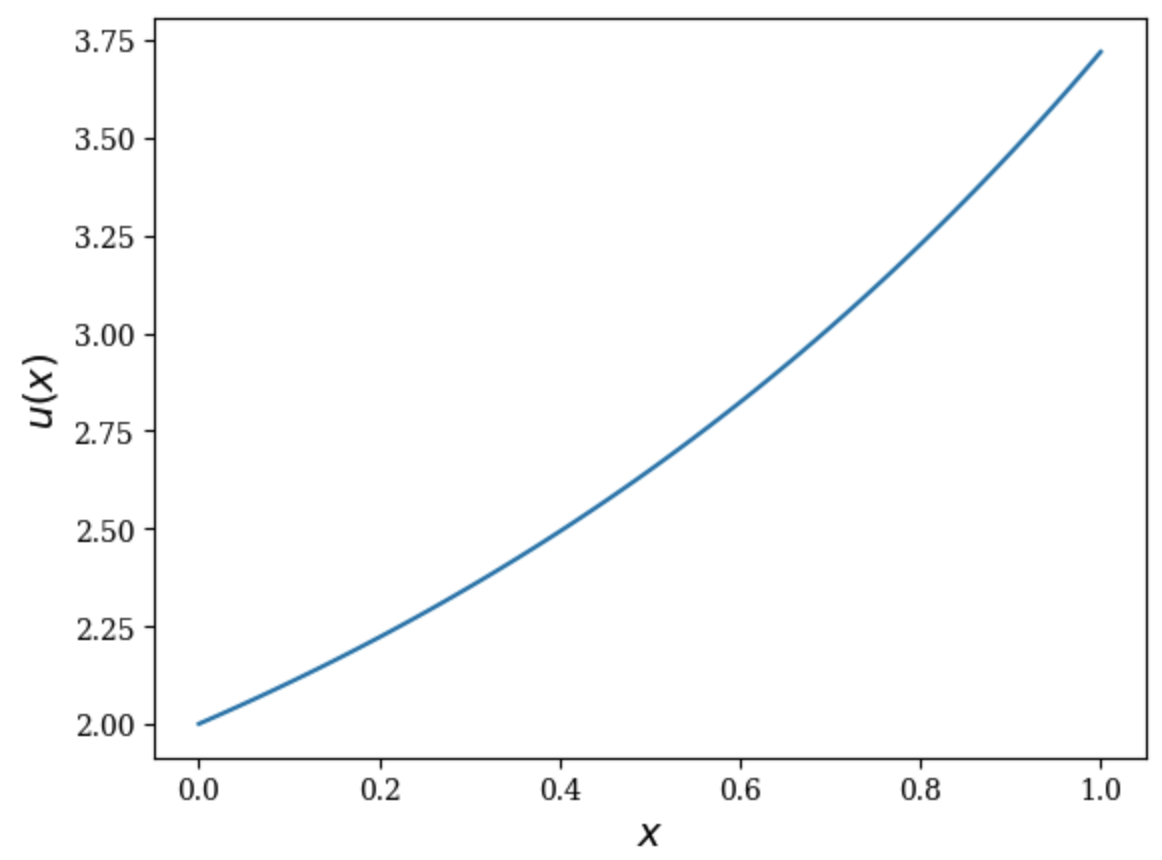}}
    \subfloat[]{\includegraphics[height=55mm, width=80mm,scale=1.0]{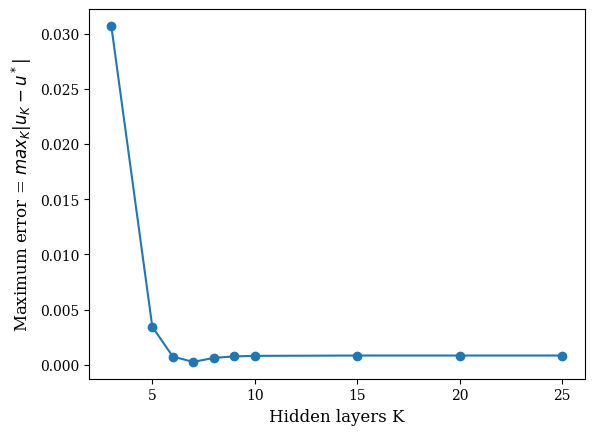}}
    \end{center}
	\caption{Results for the example \ref{ex-1}: (a) Exact solution and Fredholm NN approximation with 10 layers, (b) Maximum error as a function of the number of hidden layers.}\label{Example_2}
\end{figure}

\begin{example}\label{ex-2}
We now consider the equation: 
    \begin{eqnarray}\label{ex_2}
	f(x) = sin(x) + \int_0^{\pi/2} sin(x)cos(y)f(y) dy.
    \end{eqnarray}
One can check that the integral operator is non-expansive. For this example, we have:
    \begin{eqnarray}\nonumber
        f_n(x) = \big(2 - \frac{1}{2^n}\big)sin(x),
    \end{eqnarray}
    with $f_0(x) = sin(x)$, obtaining the fixed point solution $f(x) = \lim_{n \rightarrow \infty} f_n(x) = 2sin(x)$. 
    \par To construct the Fredholm NN, we consider a discretization of the $y$-grid with $y_0 = 0, y_N = \frac{\pi}{2}$ and $N = 2000$ points for every layer. For the first layer we have: 
	\begin{eqnarray}
		\begin{gathered}
				W_1=\left(\begin{array}{ccc}
				sin(y_1), \dots, sin(y_{N})
			\end{array}\right)^{\top}, 
			b_1=\left(\begin{array}{ccc}
			0, \dots, 0\end{array}\right).^{\top}
		\end{gathered}
	\end{eqnarray}
From the form of the kernel, the weights and biases of the next layers $W_i \in \mathbb{R}^{2000 \times 2000}$ and $b_i \in \mathbb{R}^{2000}$ for $i = 2, \dots, K$ are defined as:
\begin{eqnarray}
\begin{gathered}
	W_i=\left(\begin{array}{ccc}
		sin(y_1)cos(y_1) & \ldots & sin(y_{N})cos(y_1) \\
		\vdots & \ldots &\vdots \\
		sin(y_1)cos(y_{N}) & \ldots & sin(y_{N})cos(y_{N})
	\end{array}\right),
b_i=\left(\begin{array}{ccc}
	sin(y_1), \dots, sin(y_{N}) \end{array}\right)^{\top}.
\end{gathered}
\end{eqnarray}
We consider this architecture with a $y-$grid of size $N=2000$, with $y_0 = 0$ and $y_N = \pi/2$, for every layer and a total of $M=15$ layers. The true solution for $x\in [0, 2\pi]$ along with the resulting DNN and the corresponding errors are shown in Fig. \ref{Example_3}.
\end{example}

\begin{figure}[!ht]
    \begin{center}
    \subfloat[]{\includegraphics[height=55mm, width=80mm,scale=1.0]{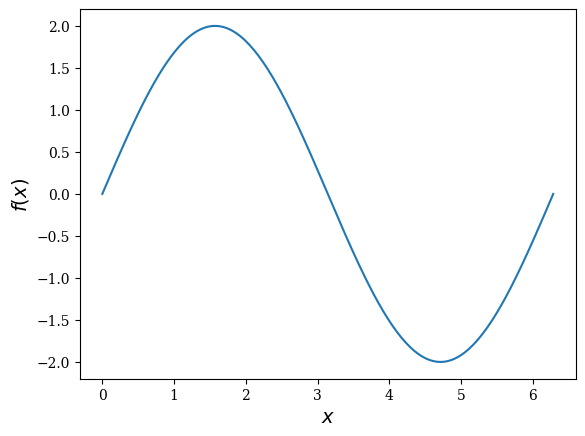}}
    \subfloat[]{\includegraphics[height=55mm, width=80mm,scale=1.0]{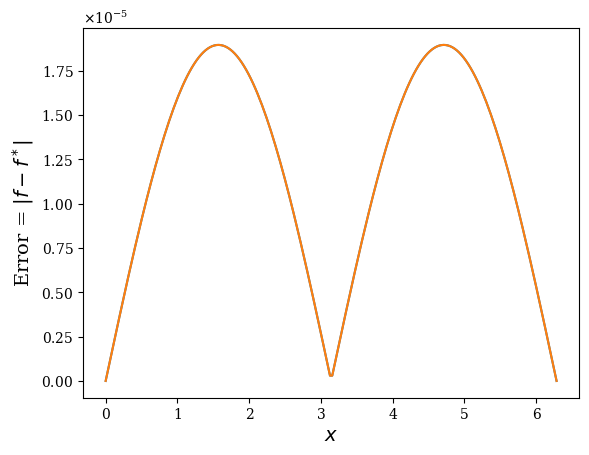}} \\
    \subfloat[]{\includegraphics[height=55mm, width=80mm,scale=1.0]{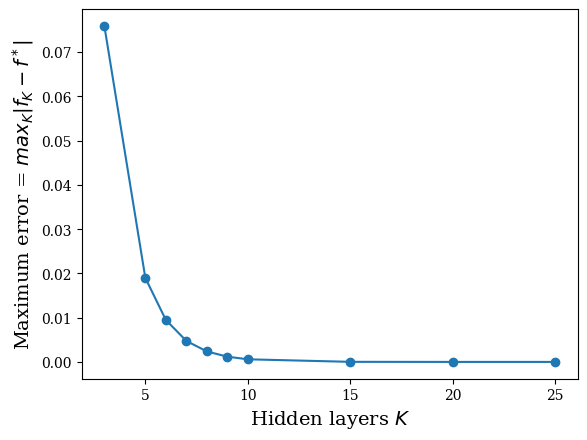}}
    \end{center}
	\caption{Results for the example \ref{ex-2}: (a) Exact solution and Fredholm NN approximation with 15 layers (b) Error in the Fredholm NN (c) Maximum error with respect to the number of hidden layers.}\label{Example_3}
\end{figure}

\subsection{Boundary Value problems}
Here, we will be focusing  on boundary value problems (BVP); for an in-depth analysis of the connection between such Boundary Value Problems and IEs we refer the reader to \cite{wazwaz2011linear}. The required connection is established via the following result:
\begin{lemma}\label{ode-ie}
    Consider a BVP of the form:
    \begin{eqnarray}
    y''(x) + g(x)y(x) = h(x), \,\,\,\, 0<x<1,
    \end{eqnarray}
    with $y(0) = \alpha, y(1) = \beta$.
Then we can solve the BVP by obtaining the following FIE:
\begin{eqnarray}
    u(x) = f(x) + \int_{0}^{1} K(x,t) u(t)dt,
\end{eqnarray}
where:
\begin{eqnarray} \notag
    u(x) = y''(x), \\
    f(x) = h(x) - \alpha g(x) - (\beta - \alpha) x g(x), \notag
\end{eqnarray} 
and the kernel is given by:
\begin{eqnarray}\label{ode-kernel}
K(x,t) = 
    \begin{cases}
        t(1-x)g(x), \,\,\, 0 \leq t \leq x \\
        x(1-t)g(x), \,\,\, x\leq t \leq 1.
    \end{cases}
\end{eqnarray}
Finally, by definition of $u(x)$, we can obtain the solution to the BVP by:
\begin{eqnarray}\label{transf}
    y(x) = \frac{h(x) - u(x)}{g(x)}.
\end{eqnarray}
\end{lemma}

\begin{example}\label{ode-1}
Consider the following BVP depending on a parameter $p$ given by:
    \begin{equation}
        y''(x) + \frac{3p}{(p+x^2)^2}y(x) = 0,
    \end{equation}
    for $ x \in [0, 1]$ with $y(0) = 0$ and $y(1) = 1/\sqrt{p + 1}$.
The analytical solution is known to be:
\begin{eqnarray}
    y(x) = \frac{x}{\sqrt{p + x^2}}.
\end{eqnarray}
We can obtain the FIE according to Lemma \ref{ode-ie}, with:
\begin{flalign}\label{weights-ode}
    g(x) = \frac{3p}{(p+t^2)^2}, \,\,\, h(x) = 0.
\end{flalign}
As in the examples above, we consider a discretization of the $t$-grid and the weights and biases become: 
\begin{eqnarray}\label{weights-ode-2}
\begin{gathered}
	W_i=\left(\begin{array}{ccc}
		K(t_1,t_1) & \ldots & K(t_{N},t_1)\\
		\vdots & \ldots &\vdots \\
		K(t_1,t_{N}) & \ldots & K(t_{N},t_{N})
	\end{array}\right),\,\,
b_i=\left(\begin{array}{ccc}
	f(t_1), \dots, f(t_{N}) \end{array}\right)^{\top}.
\end{gathered}
\end{eqnarray}
In this example we set $p = 3.2$ to ensure that the integral operator is a contraction. We construct the Fredholm NN with 10 hidden layers that estimates $u(x)$ across the discretized interval $[0,1]$ and subsequently transform the estimate to the solution of the BVP by applying (\ref{transf}). Fig. \ref{Example_4} displays the graph of both the analytical and the Fredholm NN solution, and the corresponding error. For completeness, we also show the maximum error as a function of the number of layers used for the Fredholm NN, illustrating that the error sufficiently converges for $M=10$.
\begin{figure}[!ht]
	\begin{center}\subfloat[]{\includegraphics[height=55mm, width=80mm,scale=1.0]{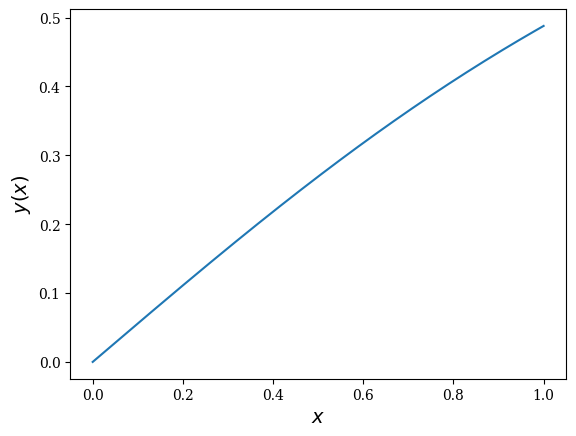}}
    \subfloat[]{\includegraphics[height=57mm, width=80mm,scale=1.0]{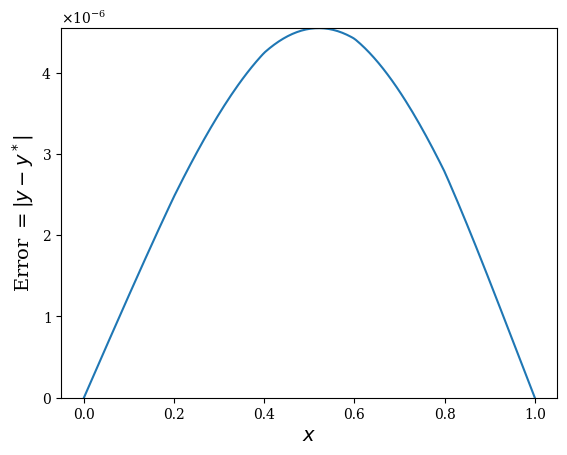}}	
	\end{center}
	\caption{Results for the example \ref{ode-1}: (a) Exact and 10-layer Fredholm NN approximation of the solution to the IE (b) Error of the Fredholm NNs approximation of $y(x)$.}\label{Example_4}
\end{figure}
\end{example}

\begin{example}\label{ode-2}
    In this final example of linear problems, we consider the BVP given by:
    \begin{equation}
        y''(x) + xy(x) = 0, 
    \end{equation}
for $ x \in [0, 1]$ and with boundary conditions $y(0) = 0, y(1) = 2$.
With $g(x) = x, f(x) = -2x^2$ and the kernel as defined above, it is straightforward to obtain the corresponding FIE, with the kernel given by (\ref{ode-kernel}).
The analytical solution of this ODE is written in terms of the Airy functions and is given by:

\begin{eqnarray}
    y(x) = \frac{2\big(\sqrt{3}Ai(\sqrt[3]{-1}x)- Bi(\sqrt[3]{-1}x)\big)}{Ai(\sqrt[3]{-1})- Bi(\sqrt[3]{-1})}.
\end{eqnarray}

Using the Fredholm NN and the final calculation (\ref{transf}) gives the results shown in Fig. \ref{Example_5}. 
\end{example}
\begin{figure}[ht]
 	\begin{center} \subfloat[]{\includegraphics[height=55mm, width=80mm,scale=1.0]{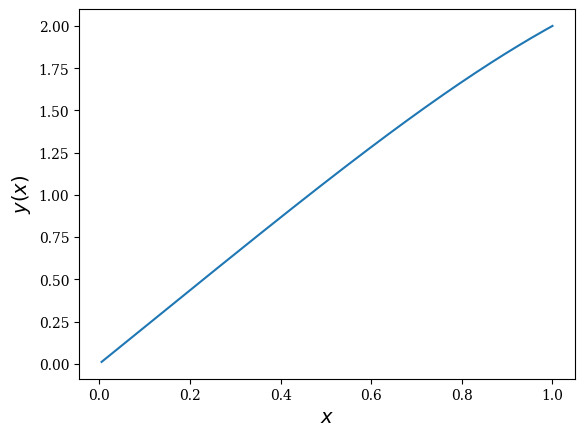}}
    \subfloat[]{\includegraphics[height=57mm, width=80mm,scale=1.0]{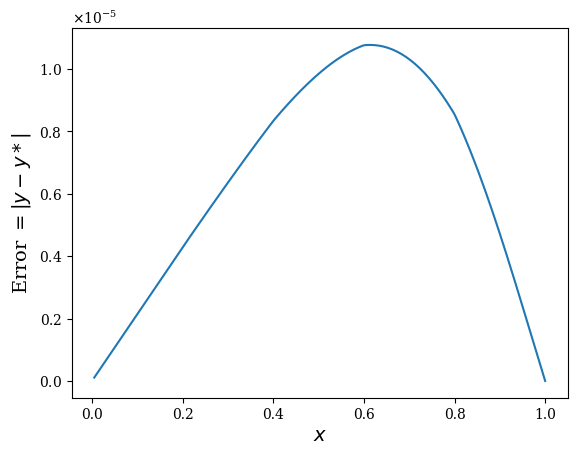}}	
	\end{center}
	\caption{Results for the example \ref{ode-2}: (a) Exact and 15-layer Fredholm NN approximation of the solution to the IE (b) Error of the Fredholm NN approximation of $u(x)$.}\label{Example_5}
\end{figure}

\subsection{Non-linear Fredholm Integral Equations}
Finally, we consider examples for non-linear problems, for which we will apply Algorithm \ref{alg:nl-fnn}, as described above. 
\begin{example}\label{nl-1}
Solution of the non-linear FIE: 
\begin{eqnarray}
    u(x) = \log(x) + \frac{143}{144} + \frac{1}{36} \int_0^1 t u^2(t)dt.
\end{eqnarray}
The analytical solution is $u(x) = \ln(x) +1$. We apply Algorithm \ref{alg:nl-fnn} with $N' = 5$ iterations, and where each linear IE is approximated with a Fredholm NN with $M = 7$ hidden layers using a grid with $N = 1000$ points. The resulting approximation is shown in Fig. \ref{ex-nl-1-fig}, and the corresponding error is $\norm{u_{N'} - u^*} \approx 8.34 \cdot 10^{-6}$.
\end{example}

\begin{example}\label{nl-2}
Solution of the non-linear FIE: 
\begin{eqnarray}
    u(x) = sin(x)+1-\frac{\pi}{12} - \frac{5\pi^2}{144} + \int_0^\pi t(u(t)+u^2(t))dt.
\end{eqnarray}
The analytical solution is $u(x) = \sin(x) + 1$. The iterative Algorithm is applied again with $N'=7, M=7$ and $N=2000$. The approximation is given in Fig. \ref{ex-nl-2-fig} with $\norm{u_{N'} - u^*} \approx 2.21 \cdot 10^{-3}$.
\end{example}

\begin{example}\label{nl-3}
Solution of the non-linear FIE: 
\begin{eqnarray}
    u(x) = 2-\frac{1}{3} (2\sqrt{2}-1)x-x^2 + \int_0^1 xy \sqrt(u(y))dy.
\end{eqnarray}
The analytical solution is $u(x) = 2-x^2$. We applied the Fredholm NN, with $N' = 10, M = 10$ over a grid with $N = 3000$ points. The approximation is shown in Fig. \ref{ex-nl-3-fig} and the error is $\norm{u_{N'} - u^*} \approx  4.4 \cdot 10^{-3}$.
\end{example}

\begin{figure}[!ht]
	\begin{center} \subfloat[]{\includegraphics[height=55mm, width=80mm,scale=1.0]{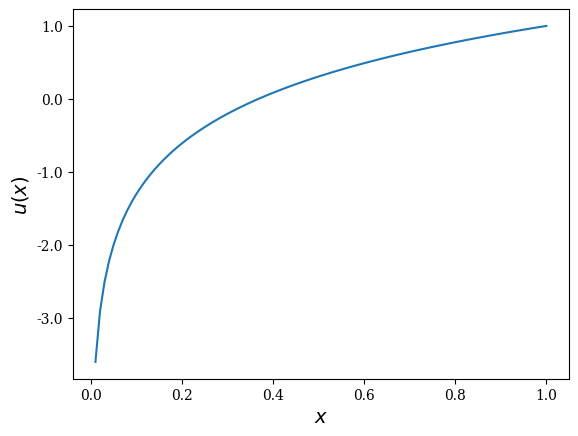} \label{ex-nl-1-fig}}
 \subfloat[]{\includegraphics[height=55mm, width=80mm,scale=1.0]{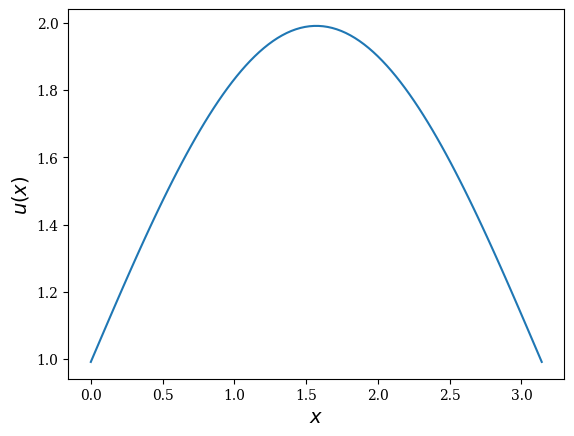}\label{ex-nl-2-fig}}\\
 \subfloat[]{\includegraphics[height=55mm, width=80mm,scale=1.0]{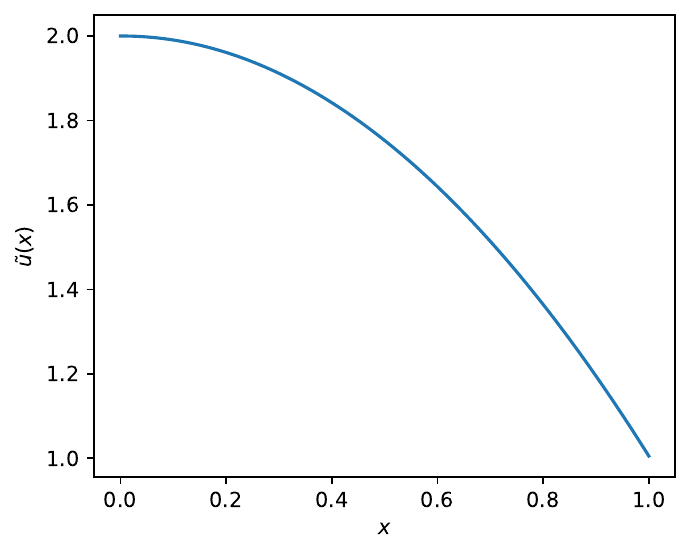} \label{ex-nl-3-fig}}
	\end{center}
	\caption{Non-linear FIEs examples: Fredholm NN result for (a) Example \ref{nl-1} (b) Example \ref{nl-2} (c) Example \ref{nl-3}.}\label{Non-linear-figs}
\end{figure}


\section{Application to elliptic PDEs: Boundary Integrals}
We now show how Fredholm NNs can be used to solve elliptic PDEs, by taking advantage of the Potential Theory, which allows us to consider integral representations of the solution to the PDEs. 
Before detailing our approach for such PDEs, we remind the reader of the main results pertaining to the Potential Theory below (see \cite{folland2020introduction} for further details and proofs).
\begin{theorem}[\cite{folland2020introduction}]\label{BIM}
    Consider the two-dimensional Laplacian equation for $u(x,y)$:
    \begin{eqnarray}\label{laplace}
        \begin{cases} u_{xx} + u_{yy} =0, & \text { for } {\bf x} = (x,y) \in \mathcal{D} \\ u(x,y)= f(x,y) & \text { for } {\bf x} \in \partial \mathcal{D}\end{cases}.
    \end{eqnarray}
The solution to the PDE can be written via the double layer boundary integral given by:
\begin{eqnarray}\label{potential-integral}
   u({\bf x}) \equiv  u(x,y) =  - \int_{\partial \mathcal{D}} \mu({\bf y}) \frac{\partial \Phi}{\partial \nu_{\bf y}}({\bf x} - {\bf y}) d \sigma_{\bf y}, \,\, {\bf x} \in \mathcal{D},
\end{eqnarray}
where $\Phi(x)$ is the fundamental solution of the Laplace equation, $\nu({\bf y})$ is the outward pointing normal vector to ${\bf y}$, and $\frac{\partial \Phi}{\partial \nu_{\bf y}} = \nu({\bf y}) \cdot \grad_{\bf y}{\Phi}$. In can be shown that the following limit holds, as we approach the boundary: 
\begin{eqnarray}\label{BIE-limit}
\lim _{\substack{{\bf x} \rightarrow {\bf x^{\star}} \\ {\bf x} \in \mathcal{D}}}   - \int_{\partial \mathcal{D}} \mu({\bf y}) \frac{\partial \Phi}{\partial \nu_{\bf y}}({\bf x} - {\bf y}) d \sigma_{\bf y} =u\left({\bf x}^{\star}\right) + \frac{1}{2} \mu\left({\bf x}^{\star}\right), \quad {\bf x}^{\star} \in \partial \mathcal{D}.
\end{eqnarray}
Hence, the function $\mu({\bf x}^{\star})$, defined on the boundary, must satisfy the Boundary Integral Equation (BIE):
\begin{eqnarray}\label{BIE}
    \mu({\bf x}^{\star}) = 2f({\bf x}^{\star}) + 2 \int_{\partial \mathcal{D}} \mu({\bf y}) \frac{\partial \Phi}{\partial \nu_{\bf y}}({\bf x}^{\star} - {\bf y}) d \sigma_{\bf y}, \,\,\ {\bf x}^{\star} \in \partial \mathcal{D}.
\end{eqnarray}
\end{theorem}
The above theorem connects a FIE, namely (\ref{BIE}), with the solution to the Laplace equation. Hence, we can use the proposed Fredholm NNs to calculate the boundary function and the subsequent solution to the PDE. However, numerically, ultimately obtaining the solution requires estimating the integral (\ref{potential-integral}), which may be difficult as we approach the boundary due to the singularity of the kernel when ${\bf x} = {\bf y} \in \mathcal{D}$. However, from the mathematical theory we know that the kernel is well-defined nonetheless at the boundary and so this integral is well-defined. Hence, the question is raised regarding how we can approach the boundary without creating large errors due to the kernel blowing-up, whilst also accounting for the required correction term $\frac{1}{2}\mu({\bf y}^{\star})$ appearing in (\ref{BIE-limit}) when the boundary is finally reached. \par 
To answer this, we turn to the theory underlying the calculations of (\ref{BIE-limit}); 
here we will be using the approach and notation from \cite{folland2020introduction}. In particular, letting $P({\bf x})$ represent the potential integral in (\ref{potential-integral}), it can be shown that, as ${\bf x} \rightarrow {\bf x}^{\star}$:
\begin{eqnarray}\label{limit-calc}
P({\bf x}) = I({\bf x})- I({\bf x}^{\star})+\frac{1}{2} \mu({\bf x}^{\star}) + P({\bf x}^{\star}), 
\end{eqnarray}
where the term $I({\bf x})$ arises in the analytical calculation of the limit as we approach the boundary and is given by:
\begin{eqnarray}\label{i-term}
    I({\bf x}) = - \int_{\partial \mathcal{D}}\big(\mu({\bf y}) - \mu({\bf x}^{\star})\big) \frac{\partial \Phi}{\partial \nu_{\bf y}}({\bf x} - {\bf y}) d \sigma_{\bf y}.
\end{eqnarray}
Therefore, (\ref{limit-calc}) can be rewritten as:
\begin{eqnarray}\label{limit-nn}
    P({\bf x}) =  - \int_{\partial \mathcal{D}}\big(\mu({\bf y}) - \mu({\bf x}^{\star})\big) \Big(\frac{\partial \Phi}{\partial \nu_{\bf y}}({\bf x} - {\bf y})- \frac{\partial \Phi}{\partial \nu_{\bf y}}({\bf x}^{\star} - {\bf y})  \Big)d \sigma_{\bf y} +\frac{1}{2} \mu({\bf x}^{\star}) + P({\bf x}^{\star}).
\end{eqnarray}
Intuitively, notice now that this formulation creates a smoothing effect, since as ${\bf x} \rightarrow {\bf x}^{\star}$, the term $I({\bf x}) - I({\bf x}^{\star})$ includes the kernel that blows up, but this is multiplied by the difference in the boundary function $\mu({\bf y}) - \mu({\bf x}^{\star})$, that approaches zero.
In this way, we are able to re-write the non-integrable potential as the sum of the integral of a smooth function and the potential $P({\bf x}^{\star})$ and a correction term on the boundary, which are all well-defined and straightforward to approximate numerically. \par 
Therefore, again connecting integral equation discretization with the fully-connected Fredholm NN, we present the following result which will be used to solve the Laplace equation.

\begin{proposition}\label{prop-laplace}
The Laplace PDE (\ref{laplace}) can be solved using a Fredholm NN with $M+1$ hidden layers, where the first $M$ layers solve the BIE (\ref{BIE}) on a discretized grid of the boundary, ${\bf y}_1, \dots, {\bf y}_N$. The final hidden and output layer are constructed in accordance to (\ref{limit-nn}), with weights $W_{M+1} \in \mathbb{R}^{N \times N}, W_O \in \mathbb{R}^N$ given by:
\begin{eqnarray}
    W_{M+1}= I_{N \times N},
    \,\,\,\,\
    W_{O}= \left(\begin{array}{cccc}
	\Delta \Phi({\bf x}, {\bf y}_1)\Delta \sigma_{\bf y}, & \Delta \Phi({\bf x}, {\bf y}_2)\Delta\sigma_{\bf y}, & \dots, & \Delta \Phi({\bf x}, {\bf y}_N) \Delta \sigma_{\bf y}
\end{array}\right)^{\top},
\end{eqnarray}
where, we define $\Delta \Phi({\bf x}, {\bf y}_i):= -\Big(\frac{\partial \Phi}{\partial \nu_{\bf y}}({\bf x} - {\bf y}_i)- \frac{\partial \Phi}{\partial \nu_{\bf y}}({\bf x}^{\star} - {\bf y}_i)\Big)$, for convenience. The corresponding biases $b_{M+1} \in \mathbb{R}^{N}$ and $b_O \in \mathbb{R}$ are given by:
\begin{eqnarray}
   b_{M+1} = \left(\begin{array}{ccc}
		- \mu({\bf x}^{\star}), \dots, - \mu({\bf x}^{\star})
	\end{array}\right)^{\top}, \,\,\,\  
 b_O= \frac{1}{2} \mu({\bf x}^{\star}) + P({\bf x}^{\star}).
\end{eqnarray}
\end{proposition}
\begin{proof}
Consider a Fredholm NN as presented in Lemma \ref{DNN_construction} with output the solution of the BIE (\ref{BIE}), $\mu({\bf y}_i)$, for $i = 1, 2, \dots N$. Then, knowing that (\ref{limit-nn}) produces the solution to the PDE, we consider a discretization of the integral across the boundary grid:
\begin{eqnarray}
P({\bf x}) =  - \sum_{i=1}^N \big(\mu({\bf y}_i) - \mu({\bf x}^{\star})\big) \Big(\frac{\partial \Phi}{\partial \nu_{\bf y}}({\bf x} - {\bf y}_i)- \frac{\partial \Phi}{\partial \nu_{\bf y}}({\bf x}^{\star} - {\bf y}_i)  \Big) \Delta \sigma_{\bf y} +\frac{1}{2} \mu({\bf x}^{\star}) + P({\bf x}^{\star}),
    \end{eqnarray}
    and it is straightforward to show that, given the $N-$dimensional output $\mu({\bf y}_i), i = 1, \dots, N$ from the FNN, the expression above can be written as an additional hidden layer and output node with the given weights and biases. 
\end{proof}

We will now apply the result above to an example over the unit disc. As we will see the approach described in Proposition \ref{prop-laplace} is able to provide very accurate results. We also compare this to a standard Finite Elements scheme.

\begin{example}
 Consider the Laplace equation (\ref{laplace}) on the unit disc, $\mathcal{D} = \{x,y \in \mathbb{R}: x^2 + y^2 \leq 1 \}$ with boundary condition $f(x,y):=x^2 - y^2 + 1$ for ${\bf x} \in  \partial \mathcal{D}$. Then, 
 \begin{eqnarray}\label{potential-integral-example}
   u({\bf x}) \equiv  u(x,y) =   - \int_{\partial \mathcal{D}} \mu({\bf y}) \frac{\partial \Phi}{\partial \nu_{\bf y}}({\bf x} - {\bf y}) d \sigma_{\bf y} = \int_{\partial \mathcal{D}} \frac{{\bf x} \cdot ({\bf x} - {\bf y})}{| {\bf x} - {\bf y}|^2}\mu({\bf y}) d \sigma_{\bf y}, \,\,\, {\bf x} \in \mathcal{D},
\end{eqnarray}
 since it is known that the fundamental solution is given by $\Phi({\bf x} - {\bf y})= - \frac{1}{2\pi}\ln |{\bf x} - {\bf y}|$. With the given parameters of the problem, we can rewrite the above in polar coordinates:
\begin{eqnarray}\label{affine}
    u(r,\phi) = \int_{0}^{2 \pi} \frac{1}{2 \pi} \frac{cos \theta (cos \theta - r cos \phi) +sin \theta (sin \theta - r sin \phi)}{(cos \theta - r sin \phi)^2 + (sin \theta - rsin \phi)^2} \mu(\theta) d \theta,
\end{eqnarray}
where the boundary function $\mu(\theta)$ satisfies the FIE:
\begin{flalign} \label{BIE-laplace}  
\mu(\phi)
= 2f(\phi) + \int_0^{2 \pi} -\frac{1}{2 \pi} \mu(\theta) d \theta,
\end{flalign}
with $f(\phi) = 1+ 2cos(2\phi)$. By discretizing the BIE on the $\theta$-grid, $\{\theta_j\}_{j \in \mathcal{J}}$, we get:
\begin{flalign}
    \mu_{m+1}(\theta_i)  = 2 \kappa f(\theta_i) + \sum_{j \in \mathcal{J}} \mu_m(\theta_j)\Big( -\frac{\kappa \Delta\theta}{2 \pi} + (1-\kappa)\mathbbm{1}_{\{j = i\}} \Big).\label{km-int}
\end{flalign}
In this form, the first term in (\ref{km-int}) can be used as the bias term and the coefficient of $\mu_m(\theta_j)$ as the weight in the construction of the Fredholm NN in accordance to Theorem \ref{FUA}, i.e.,:
\begin{eqnarray}\label{weights-km}
\begin{gathered}
	W_i=\left(\begin{array}{cccc}
		-\frac{\kappa \Delta\theta}{2 \pi} + (1-\kappa) & -\frac{\kappa \Delta\theta}{2 \pi} & \dots & -\frac{\kappa \Delta\theta}{2 \pi} \\
  -\frac{\kappa \Delta\theta}{2 \pi}  & -\frac{\kappa \Delta\theta}{2 \pi} + (1-\kappa) & \dots & -\frac{\kappa \Delta\theta}{2 \pi} \\
		\vdots & \vdots &\vdots & \vdots \\
		-\frac{\kappa \Delta\theta}{2 \pi} & -\frac{\kappa \Delta\theta}{2 \pi} & \ldots & -\frac{\kappa \Delta\theta}{2 \pi} + (1-\kappa)
	\end{array}\right),\,\, 
 \vspace{0.4cm}
 \\
b_i=\left(\begin{array}{ccc}
	2 \kappa f(\theta_1), \dots, 2 \kappa f(\theta_{N}) \end{array}\right)^{\top}.
\end{gathered}
\end{eqnarray}
The Fredholm NN provides the solution to the boundary function along the pre-defined $\theta-$grid, as shown in the first component of Fig. \ref{fig:BIE-NN}. For this example, we have constructed the Fredholm Neural Network with a grid size of $N=2000$, and $M=15$ hidden layers. 
\par Having the approximation of the boundary function, we now need to perform the calculation given by (\ref{limit-nn}) to obtain the approximation of $u(r,\phi)$. Applying Proposition \ref{prop-laplace}, we construct the final components of the Neural Network (as displayed in Fig. \ref{fig:BIE-NN}). 
\begin{figure}[ht]
    \centering
    \includegraphics[height=8cm, width=16.5cm,scale=1.0]{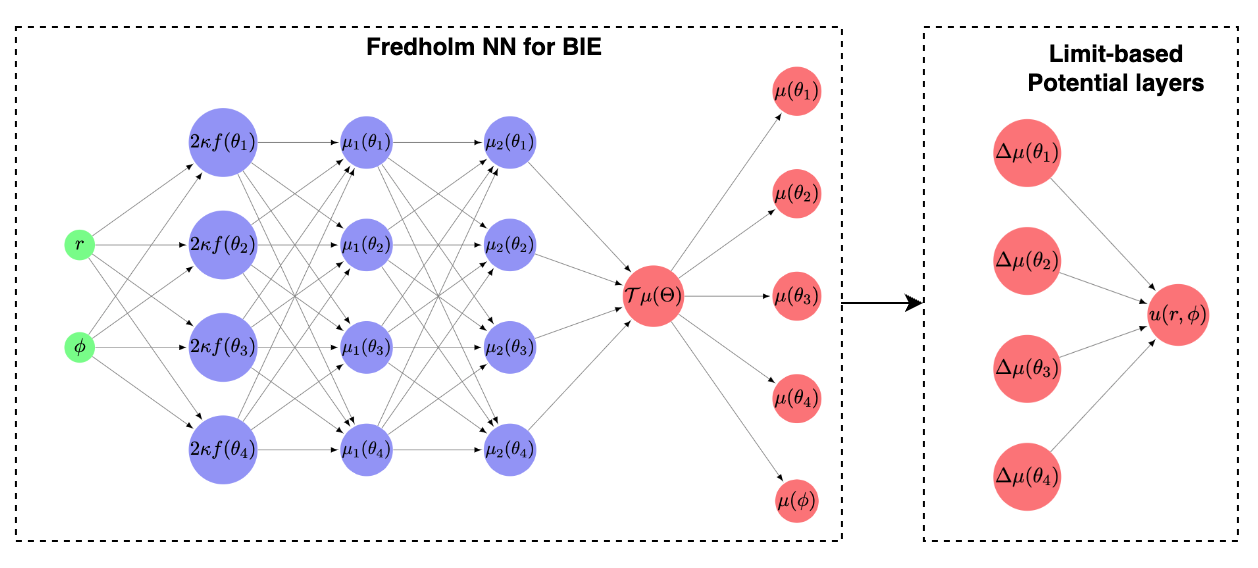}
    \caption{Schmematic of the Fredholm NN for approximating the solution of the Laplacian. The first component consists of the Fredholm NN solving the BIE. In the second component, where we have defined $\Delta \mu(\theta_i):= \big(\mu(\theta_i) - \mu(\phi)\big)$, the resulting boundary function is used to construct the final layers that simulate the limit calculation described in (\ref{limit-nn}).} \label{fig:BIE-NN}
\end{figure}
The results of the proposed approach are gathered in Fig. \ref{laplace_example}. In particular, from Fig. \ref{DNN-err} we can see that we are able to obtain extremely accurate approximations of the solution, with the maximum error across both the domain and boundary being $\norm{u_{NN} - u} = 2.27 \cdot 10^{-7}$. In particular, we see that within this framework we are able to achieve highly accurate results near and on the boundary, which is often an area where issues arise in Deep Learning. As seen, this is achieved by connecting the Fredholm NN with calculations related to Potential Theory and the construction of the Boundary Integral Equation for the PDE. 

\begin{figure}[htp]
    \begin{center}
    \subfloat[]{\includegraphics[width = 0.48\textwidth]{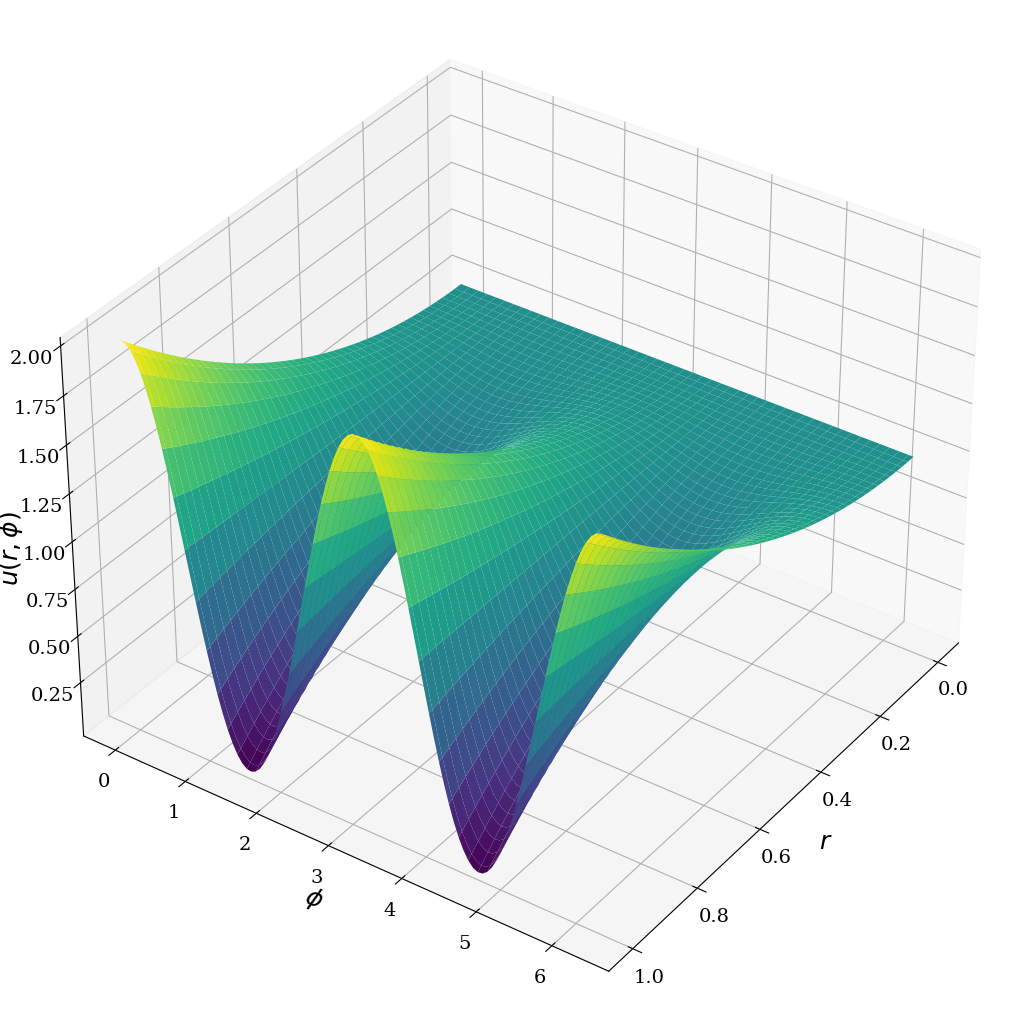}\label{DNN-sol}} 
    \subfloat[]{\includegraphics[width = 0.48\textwidth]{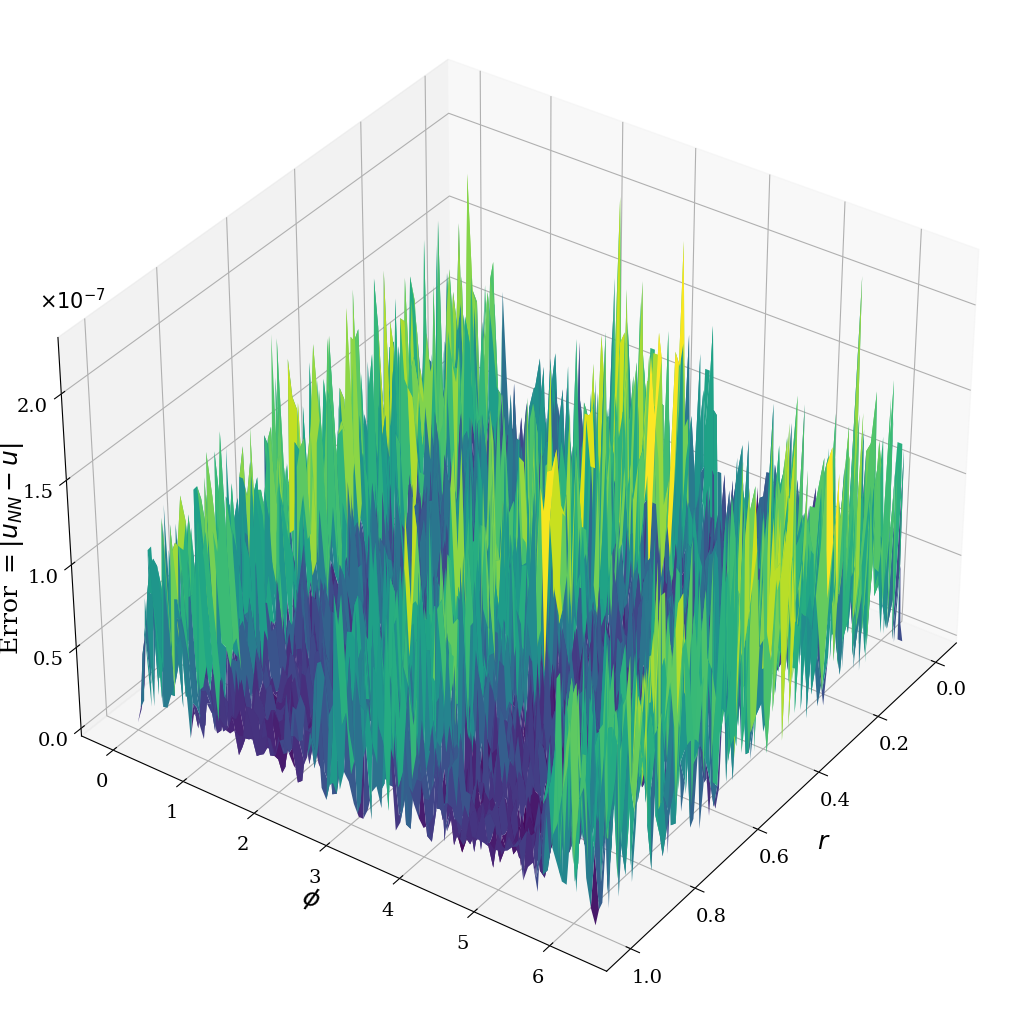}\label{DNN-err}} \end{center} 
	\caption{Results of the Fredholm NN method for the Laplacian: (a) Fredolm NN solution of the Laplacian (b) Error in the Fredholm NN approximation.}\label{laplace_example}
\end{figure}

\begin{figure}[htp]
    \centering
     \includegraphics[width = 0.50\textwidth]{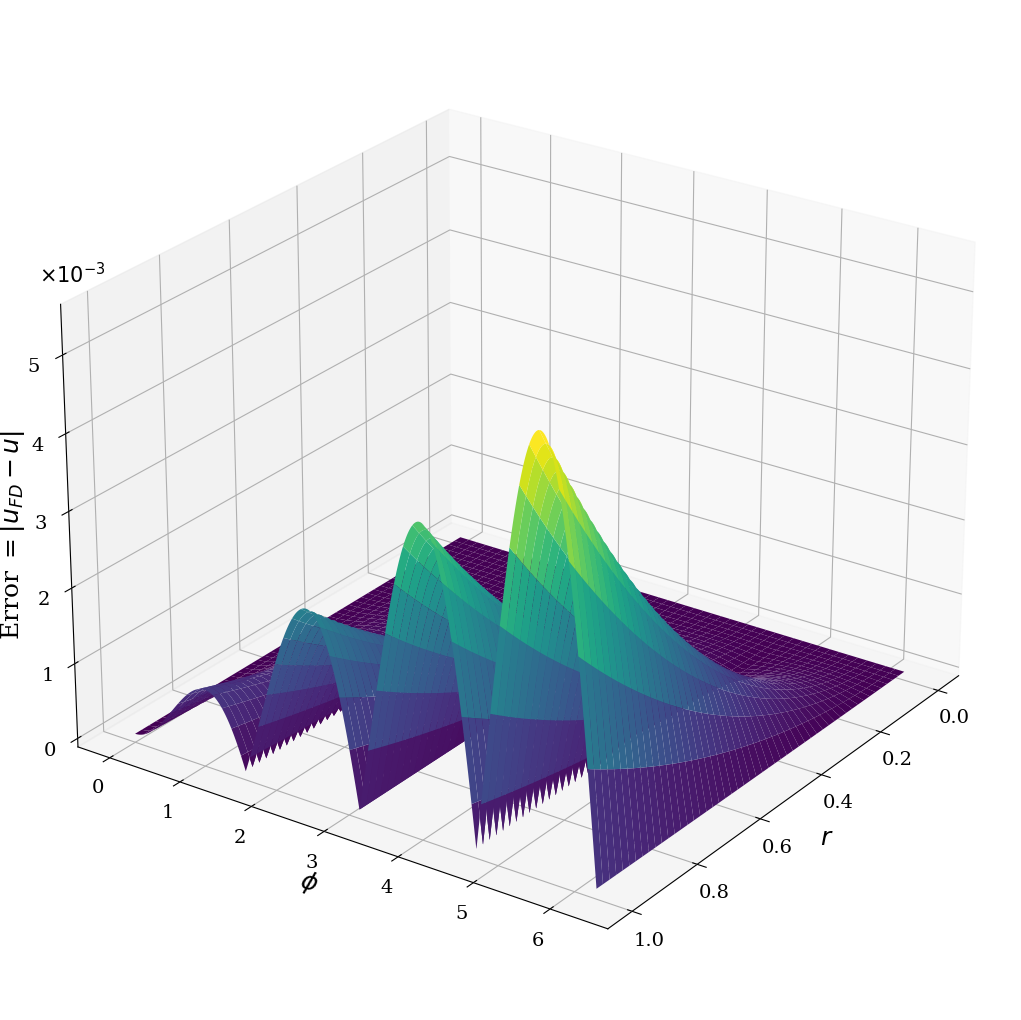}
    \caption{Error in the Finite Difference approximation.}
    \label{FD_err}
\end{figure}
\end{example} 

\par Finally, Fig. \ref{FD_err} shows the resulting absolute error using the Finite Difference approximation, where we have considered a $2000 \times 2000$ grid for the radial and angular variables and the appropriate periodic and radial boundary conditions. As shown, we reach an error of $\norm{u_{FD} - u} = 5.0 \cdot 10^{-3}$.  Hence, the proposed approach outperforms the FD scheme, whilst providing explainability for the underling DNN. Indeed, the entire construction of the DNN uses forms from the Fredholm and Potential Theory that allow us to construct the model in a way that produces accurate results both within the domain, and on the boundary (which has often been an issue with DNN based solutions of PDEs). 

\section{Fredholm Neural Networks in Inverse Problems} \label{inv-sec}
Here, we demonstrate how the proposed approach can be used to solve the inverse problem for to IEs. This setting consists of taking as data two functions $f, g : \mathcal{D} \to {\mathbb R}$ and looking to model the unknown kernel function $K : \mathcal{D} \times \mathcal{D} \to {\mathbb R}$, such that the function $f$ satisfies the integral equation (\ref{ie}) or (\ref{nl-ie-def}).

Note that 
solving the inverse problem for the unknown kernel $K$ is an ill-posed problem, not admitting a unique solution. Indeed, there are infinitely many kernels that produce the same integrated result. 
Hence, the solution of the inverse problem amounts to learning the integral operator, rather than the kernel per se.
\par Here, with this perspective, we assume that we are given data $\{\tilde{f}(x_i)\}_{k = 1, \dots, M}$ approximating the value of the solution $f$ of a FIE of the form \eqref{ie}, on some collocation points $\{x_i\}_{i=1, \dots N}, x_i \in \mathbb{R}$. These data correspond to the solution of a physical problem that can be in principle, modeled as an IE in the form of \eqref{ie}, with the kernel being unknown. For our illustrations, we assume that we know $g:\mathbb{R} \rightarrow \mathbb{R}$. Hence, the inverse problem consists of, given the data $\tilde{f}$, $g$, modelling the unknown kernel function $K$, such that it provides a consistent finite‐dimensional approximation of the integral operator, i.e., such that 
the integral equation whose solution $f$ coincides with $\tilde{f}$ on the chosen collocation points.
\par As possible choices for the kernels we will consider functions generated by a neural network a two-dimensional input and parameters $\theta$, hereinafter represented as $\hat{K}_{\theta}(x,y)$. (Note that the methodology that follows can naturally be used when considering any other type of models, as well.) Then, 
our strategy for solving the inverse problem takes advantage of the structure and convergence properties of the Fredholm NN as follows: select a set of parameters $\theta$ such that when constructing the estimated kernel $\hat{K}_{\theta}(\cdot, \cdot)$ and then feeding this into the Fredholm NN with $M$ hidden layers, denoted by $FNN_M( \cdot ; \hat{K}_{\theta})$, the output of the Fredholm NN, $\hat{f}(x;\hat{K}_{\theta})$, is as ''close'' as possible (in terms of an appropriately chosen loss function) to the given data $\tilde{f}$. The operator learning problem, then reduces to the problem of tuning the parameters $\theta$ of $\hat{K}_{\theta}$ appropriately, in terms of an optimization problem. 
 \par For our approach, we consider a simple Tikhonov regularization on the learned kernel $K_{\theta}$. 
 The complete loss function is given by: 
\begin{equation}\label{inverse-loss}
     L(\theta) = \frac{1}{N} \sum_{i=1}^{N} \Big(f(x_i) - \hat{f}(x_i; \hat{K}_\theta) \Big)^2 + \lambda_{reg}\mathcal{R}(\theta),
\end{equation}
where recall $\hat{f}(x_i;\hat{K}_{\theta})$ represents the output of the Fredhom NN and $\mathcal{R}(\theta) = \| K_{\theta}\|_{2}$. This approach is detailed in Algorithm \ref{alg:inverse} and shown schematically in Fig. \ref{fig:inverse}. 

\begin{algorithm}[hbt!]\caption{Inverse Problem using Fredholm NNs for the approximation of the unknown kernel function.}\label{alg:inverse}
\begin{algorithmic}
\State {Create grid $\{x_i\}_{i=1, \dots N}\}$ and set Fredholm Neural Network depth $M$}
\State {Set number of training epochs $N$, initialize iterations, $n=0$}
\State {Approximate the unknown kernel $K(x,y)$ using a neural network (NN), as $\hat{K}(x,y; \theta^{(n)})\equiv\hat{K}_{\theta^{(n)}}$, where $\theta^{(n)}$, denotes the set of unknown  parameters of the NN (weights and biases) at iteration $n$}
\While{$n \leq N$}
    \State {Step 1. Construct a Fredholm NN, $FNN_M(x;\hat{K}_{\theta^{(n)}})$}
    \State {Step 2. Use the Fredholm NN output, $\hat{f}(x_i; \hat{K}_{\theta^{(n)}}))$ to calculate the model loss given by \eqref{inverse-loss}.  }
    \State {Step 3. Train model and set $n=n+1$}
\EndWhile \\
\Return {$\hat{K}(x,y; \theta^{*})$}
\end{algorithmic}
\end{algorithm}
\begin{figure}[!ht]
\begin{center}
  \includegraphics[width = \textwidth]{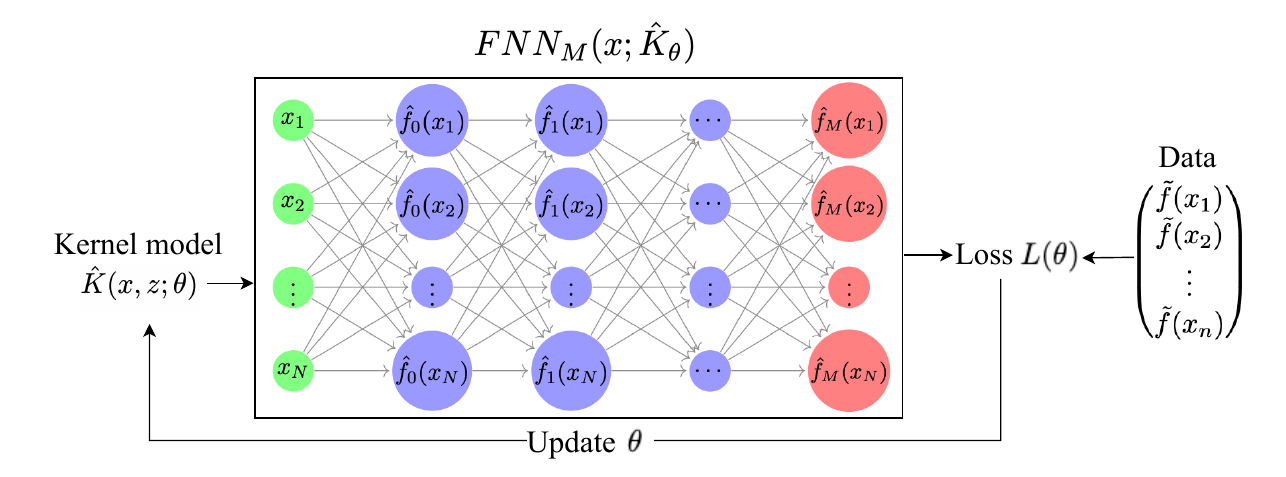} 
\end{center}
\caption{Schematic representation of Algorithm \ref{alg:inverse} to solve the inverse problem using the Fredholm NN framework.}\label{fig:inverse}
\end{figure}
\par Within this framework, the connection between the Fredholm NN and the fixed-point theory ensures that we select a kernel that converges to the fixed point given sufficient depth. The structure of the Fredholm NN itself ensures that we are simultaneously accounting for the data loss as well as ensuring that the model adheres to the physics of the problem. 

For our illustration of the methods in section \ref{inv-sec}, we consider the following example. 
\begin{example}\label{inverse-1-ex}
Consider the inverse  problem of finding the unknown  kernel function $K : [0,\pi/2] \times [0, \pi/2] \to {\mathbb R}$ such that the integral equation:
\begin{equation}\label{inverse-prob}
f(x) = \sin(x) +  \int_0 ^{1} K(x, y) f(y) dy
\end{equation}
admits the solution $f(x)=\frac{4}{3}\sin(x)$.

We consider a dataset consisting of 300 sample points generated from the forward Fredholm NN with $M=10$ layers, $\{\tilde{f}(x_i)\}_{i=1,\dots,300}$. As a model of the kernel, $\hat{K}_{\theta}(\cdot, \cdot)$, we used a shallow neural network with a single hidden layer, 20 neurons and a $\tanh$ activation function. 
\par We solve the inverse problem as described in section \ref{inv-sec}, using a Fredholm NN with $M=15$ hidden layers (a large enough depth to ensure convergence for any initial guess in the Fredholm NN), and $N = 300$ nodes per layer of the Fredholm NN (in accordance with the data points given for training) to which the learned kernel will be fitted.
For our experiments, 10 training instances of training Algorithm \ref{alg:inverse} were run, with a constant coefficient of $\lambda_{reg} = 1.0$E$-05$. We used the Levenberg–Marquardt (LM) algorithm for optimization, with a maximum of 200 iterations.
\par To assess the performance of the proposed approach, we used the learned kernel $\hat{K}_{\theta^*}$ as input to a Fredholm NN with the architecture used for training  and compared the output $\hat{f}(x,;\hat{K}_{\theta^*})$ to the original data $\tilde{f}(x)$. The results showed a high precision of the proposed method. In particular, the results of the best model over the 10 runs were as follows (median values are also given in parentheses): we achieved $MSE(\hat{f}, \tilde{f})\approx 7.32$E$-10$ $(3.66$E$-09)$ and $L_{\infty} := \max|\hat{f} - \tilde{f}| \approx 9.97$E$-05$ $(2.44$E$-04)$. In Fig. \ref{fig:inverse_example_1} we display the heat map of the learned kernel corresponding to this $MSE$ value when using the FIE regularization, along with the norm $\| \hat{f}_{i+1} - \hat{f}_i\|$ between the successive layers of the constructed $FNN$ showing the convergence to the fixed point solution.  
\begin{figure}[htp]
    \begin{center}
\includegraphics[width = 0.29 \textwidth]{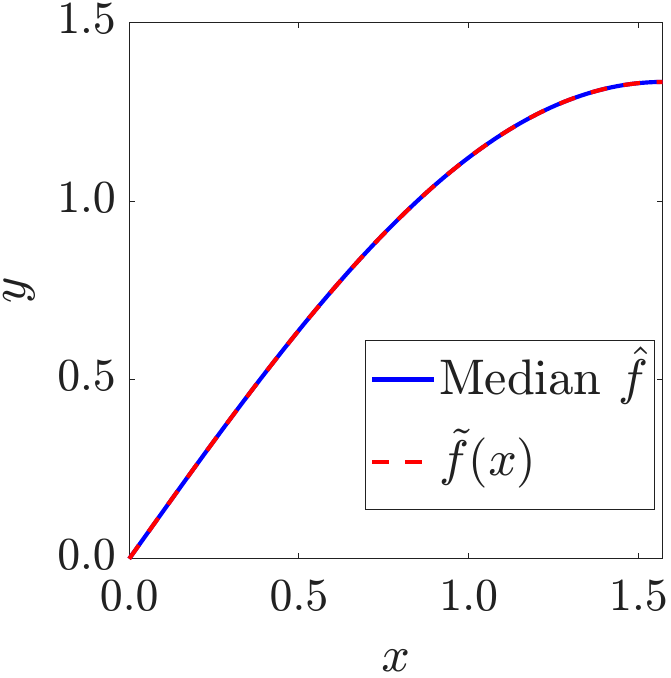}
\includegraphics[width = 0.30\textwidth]{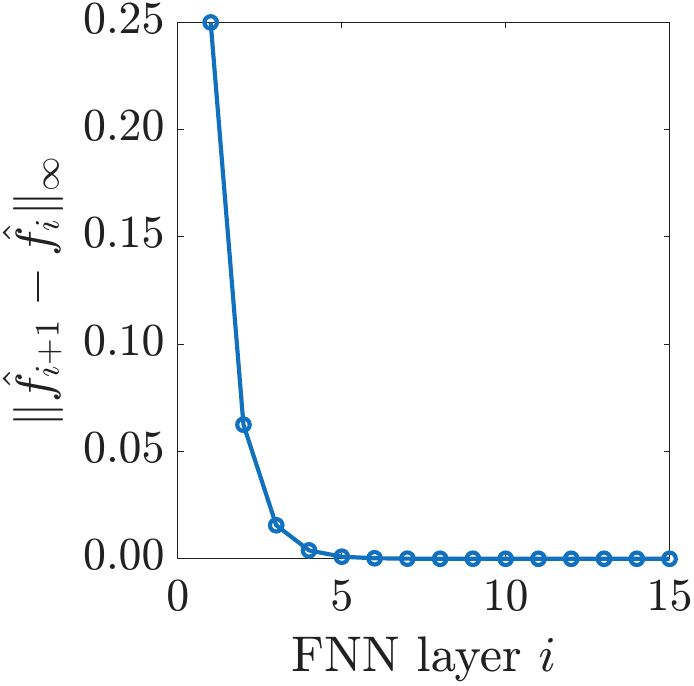}
\includegraphics[width = 0.33 \textwidth]{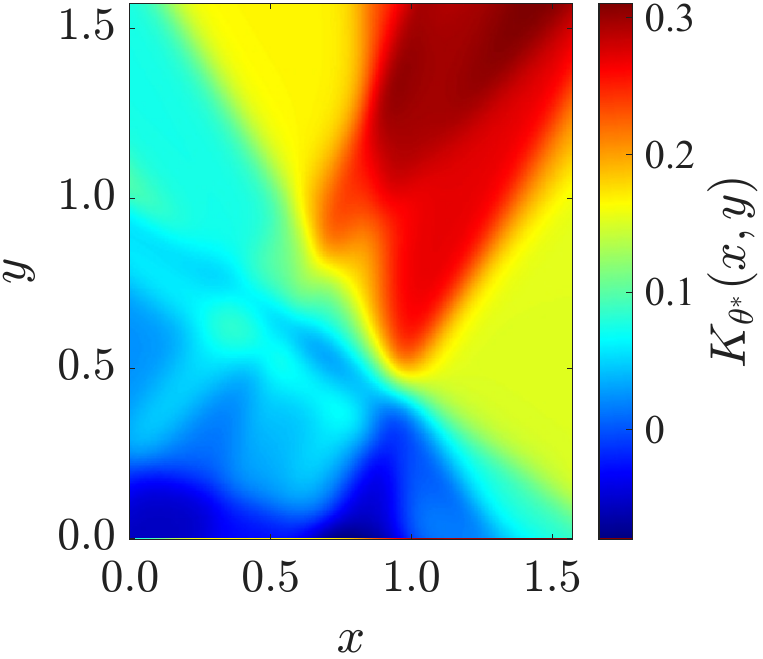}

    \end{center} 
	\caption{(Left) Median $\hat{f}(x; \hat{K}_{\theta^*})$ in comparison to the reference (training) data $\tilde{f}(x)$. (Center) The learned kernel function $\hat{K}_{\theta^*}(x,y)$. (Right) Norm between successive layers in $FNN(x; \hat{K}_{\theta^*})$, $\|\hat{f}_{i+1} - \hat{f}_i\|_{\infty}$, showing convergence to the estimated fixed-point. }\label{fig:inverse_example_1}
\end{figure}
\end{example}
\par As shown, the proposed methodology achieves high accuracy. We highlight that the Fredholm NN framework applied in this way ensures that no a-priori constraints on the initialization of the kernel model are needed, but rather the training process itself guarantees that we converge to values that satisfy the necessary properties, via passing through the Fredholm NN. In this way, we are able to learn the kernel function and obtain the contractiveness property via the implementation of the Fredholm NN.

\section{Conclusions and Discussion}
In this work, we have proposed and developed the Fredholm neural network for constructing explainable DNN models for the solution of problems that can be written as IEs. This scheme connects standard fixed point estimates to solve Integral Equations with Deep Neural Networks, providing a way to construct explainable DNN with weights, biases and number of layers arising from the theory of fixed point schemes and corresponding results related to their convergence. \par 
This approach provides an alternative to the interpretation and construction of DNNs, as training via an appropriate loss function is not required and the entire architecture can be set based on the underlying mathematical theory pertaining to the solution of IEs. We believe that this alternative viewpoint, that has started gaining attention (see for example \cite{doncevic2024recursively}) can be important in the field of scientific machine learning for numerical analysis, where the resulting models and their accompanying errors are often assessed on the basis of heuristic tests. Such heuristic DNNs lack stability, monotonicity and rigorous error bound/uncertainty quantification (UQ) analysis. With the proposed Fredholm NN, we tackle this issue, as the architecture of the model provides a-priori the error bound/UQ of the approximation. This is a step towards viewing DNNs from a numerical analysis point of view. For our illustrations, we used the proposed scheme to simple boundary value problems and elliptic PDEs. 
\par  We re-iterate that the main focus of this work is to introduce the Fredholm neural networks.  
In a following work, we will focus on the application and extension of the proposed scheme to deal with the solution of PDEs in high dimensions where lies the relative advantage of DNNs over schemes like FEMs and Boundary Elements. Moreover, a more general framework is required to consider expansive kernels, as well as IEs of different forms (e.g., Volterra integral equations, \cite{brunner2017volterra}). We defer the above for future works.\par 
In concluding, we believe that the proposed explainable DNN opens new research directions pertaining to general fields such as explainable AI (XAI, see e.g., \cite{angelov2021explainable}) and Uncertainty Quantification (see e.g., \cite{psaros2023uncertainty}), since fixed point neural networks provide a DNN architecture directly related to well-established numerical, analysis methods. One can think of this perspective as task-oriented DNNs (in analogy to \cite{fabiani2024task}) that facilitate specific numerical analysis computations such as the solution of PDEs, ODEs, solution of large scale linear and nonlinear equations \cite{doncevic2024recursively}.

\section*{Acknowledgements}
K.G. acknowledges support from the PNRR MUR Italy, project PE0000013-Future Artificial Intelligence Research-FAIR. 
C.S. acknowledges partial support from the PNRR MUR Italy, projects PE0000013-Future Artificial Intelligence Research-FAIR \& CN0000013 CN HPC - National Centre for HPC, Big Data and Quantum Computing. Also from the Istituto di Scienze e Tecnologie per l'Energia e la Mobilità Sostenibili (STEMS)-CNR. A.N.Y. acknowledges the use of resources from the Stochastic Modelling and 
Applications Laboratory, AUEB.  
\bibliographystyle{plain}
\bibliography{Fredholm_Neural_Networks}

\end{document}